\documentclass[review,hidelinks,onefignum,onetabnum]{siamart220329}
\nolinenumbers

\usepackage{graphicx}
\usepackage{amsmath}
\usepackage{amssymb}
\usepackage{epstopdf}
\usepackage{xcolor}
\usepackage{enumerate}
\usepackage{multirow}
\usepackage{bbm}
\usepackage[normalem]{ulem}
\usepackage{subfigure}

\usepackage{soul}

\newcommand{\blue}[1]{\textcolor{black}{#1}}

\newcommand{\eps}{\varepsilon}
\newcommand{\R}{\mathbb{R}}

\newcommand{\N}{\mathbb{N}}

\newcommand{\rme}{\mathrm{e}}

\newcommand{\spann}{\mathrm{span}}

\newcommand{\tO}{\tilde{\Omega} }

\newtheorem{Lemma}{Lemma}[section]
\newtheorem{Theorem}{Theorem}
\newtheorem*{Theorem**}{Theorem}
\newtheorem{Proposition}[Lemma]{Proposition}
\newtheorem{Remark}[Lemma]{Remark}
\newtheorem{Definition}[Lemma]{Definition}
\newtheorem{Hypothesis}[Lemma]{Hypothesis}

\newsiamremark{remark}{Remark}
\newsiamremark{hypothesis}{Hypothesis}
\crefname{hypothesis}{Hypothesis}{Hypotheses}
\newsiamthm{claim}{Claim}

\headers{Analysis and Simulation of a Nonlocal Gray--Scott model}{L. Cappanera, G. Jaramillo, C. Ward}

\title{Analysis and Simulations of a Nonlocal Gray-Scott Model \thanks{
\funding{This work is supported in part by the National Science
  Foundation under grant DMS-1911742 (GJ).}}}

\author{Loic Cappanera\thanks{Department of Mathematics, University of Houston, Houston TX
  (\email{lmcappan@central.uh.edu}).}
\and Gabriela Jaramillo\footnotemark[2] \thanks{Corresponding author,
  (\email{gabriela@math.uh.edu}).}
\and Cory Ward\footnotemark[2] \thanks{(\email{crywrd@gmail.com}).}}

\usepackage{amsopn}

\begin{document}
\nolinenumbers

\maketitle

\begin{abstract}
The Gray-Scott model is a set of reaction-diffusion equations that describes chemical systems far from equilibrium.
Interest in this model stems from its ability to generate spatio-temporal structures, including
pulses, spots, stripes, and self-replicating patterns. We consider an extension of this model in which
the spread of the different chemicals  is assumed to 
be nonlocal, and can thus be represented by an integral operator.
In particular, we focus on the case of strictly positive, symmetric, $L^1$ convolution kernels that
have a finite second moment.
Modeling the equations on a finite interval,
we prove the existence of small-time weak solutions in the case of
nonlocal Dirichlet and Neumann boundary constraints.
We then use this result to develop a finite element numerical scheme 
that helps us explore the effects of nonlocal diffusion on the formation of pulse solutions.

\vspace*{0.1in}
{\small
{\bf Running head:} {Analysis and Simulations of a Nonlocal Gray-Scott Model}


}
\end{abstract}

\maketitle

\section{Introduction}

Since the 1950's chemical reactions far from  equilibrium
have been a source of interesting mathematical problems.
The beautiful and often intricate spatio-temporal structures that emerge  have led 
to numerous studies related to
 pattern formation,
 not only in the context of chemical reactions but also in other physical and biological systems, 
\cite{winfree2001, crosshohenberg, hoyle2006pattern, ball2009shapes}. 
In the past, most  mathematical and numerical analyses 
of these problems
have been done under the assumption that  
 they are well described by reaction-diffusion equations.
 However, when considering biological phenomena
 a nonlocal form of diffusion often provides a better description of the
 transport processes that are involved.
 This is the case of vegetation and population models, where the spread 
 of plant seeds and individuals is better represented by a convolution operator \cite{ bullock2017, humphries2010, nathan2003, nathan2006, nathan2008, nathan2012, pueyo2008, reynolds2007, sims2008}.
 This modeling choice then leads to  integro-differential equations, which are 
 often difficult to analyze and simulate. In this paper  we take on this challenge and consider 
a {\it nonlocal} Gray-Scott model.
 We prove the well-posedness of the equations, and construct a numerical scheme to solve them using  finite elements.

Our starting point is the {\it local} Gray-Scott model, which
 describes a general autocatalytic reaction involving two chemical species, $U$ and $V$,
of the form,
\[ U + 2 V \longrightarrow 3V, \qquad V \longrightarrow P.\]
This reaction is assumed to take place inside a well-stirred isothermal open-flow reactor, 
where species $U$ is constantly supplied, and the product $P$ is always removed \cite{gray1984}.
Experiments \cite{lee1993, lee1994},  numerical simulations \cite{pearson1993, reynolds1994, doelman1997, morgan2000}, and mathematical analysis \cite{reynolds1994, doelman1997, doelman1998, morgan2000, doelman2000, doelman2001, wei2001, morgan2004, kolokolnikov2005, mcgaugh2004}, then show that under these conditions
the system exhibits periodic solutions, including spots and stripes,
self-replicating patterns, and other interesting spatio-temporal structures.

Our interest in the Gray-Scott model also comes from its connection to the generalized Klausmeier model,
a system of reaction-diffusion equations used to describe vegetation patterns in dry-land ecosystems \cite{doelman2013}.
Indeed, in \cite{doelman2013} it is shown that  by a suitable rescaling of the variables
the steady states of the Gray-Scott equations can be mapped to those of the Klausmeier model. 
It is this relationship between both systems and the fact that seed dispersal is better
 modeled using a nonlocal form of diffusion \cite{ bullock2017,  nathan2003, nathan2006, nathan2008, nathan2012, pueyo2008}, that leads us to consider 
 the following nonlocal version of the Gray-Scott model
 \begin{equation}\label{e:gs_nonlocal}
\begin{split}
 u_t= & d_u K  u - uv^2 + f(1-u), \\
 v_t= & d_v K  v + uv^2 -(f+\kappa)v,
 \end{split}
\end{equation}

 which we assume is posed on a bounded interval, $\Omega = [-L,L]$.
In the equations, the variables $u$ and $v$ represent the concentrations of the chemical species,
$U$ and $V$, respectively, while the  linear operator, $K$, represents a nonlocal form of diffusion.
In addition, the constants $d_u,d_v$, denote the diffusive rates, and the parameters
$f$ and $\kappa$ represent the feed and kill rate, respectively.

Although the  use of nonlocal diffusion in a chemical model might seem at first artificial, one can justify this choice 
by looking at chemical reactions involving a fast component.
For instance, consider a system of the form
\begin{align*}
u_t= & (w-u) - uv^2 + f(1-u), \\
 \eps w_t &= (\Delta -1) w + u
 \end{align*}
with $\eps$ representing a small parameter.
In the limit as $\eps \to 0$,  one can adiabatically eliminate the fast variable.
This means using the second equation to solve for $w$ in terms of $u$ via the Green's function, $G$, for the operator $(\Delta -1)$. When this expression is then inserted back into the first equation, the result is  the \blue{integral operator}, 
\[ (w-u)  = G \ast u -u ,\]
which can be interpreted as a nonlocal form of diffusion.
In particular, the above calculations tell us that 
because the fast component diffuses more quickly,
it is able to interact in a meaningful way with elements that are far away, thus creating long-range 
dispersal effects.

To study the role that  nonlocal diffusion has in shaping solutions,
in this paper we first 
  establish the  weak formulation of the above nonlocal Gray-Scott model and prove that 
these equations are well-posed. 
We then use this information to develop a numerical scheme to solve the evolution problem. 
We use finite elements to discretize the nonlocal operator, and a BDF scheme to 
perform the time stepping. 
Finally, we simulate the equations in order to determine how nonlocal diffusion affects single-pulse solutions.

While previous numerical studies of  nonlocal Gray-Scott models assume that long-range 
dispersal is described by a fractional Laplacian 
\cite{wang2019, mostafa2019, liu2020, owolabi2020, mostafa2022, han2022}, 
in this article we focus on 
\blue{integral} operators described by 
spatially-extended, positive, and symmetric $L^1$ convolution kernels.
We also assume nonlocal homogeneous Dirichlet and Neumann boundary constraints following the 
approach taken by Du et al in \cite{du2012, du2013}.
 In the Dirichlet case 
this means that the values of $u$ and $v$ are set to zero outside the computational domain, $\Omega$,
while in the Neumann case  the conditions 
 $K  u = 0,$ and $K  v =0$,  are enforced in $\Omega^c$. 
Our mathematical analysis and formulation of the numerical scheme is also influenced by 
the theory developed in these last two references. However, while these works
focus on kernels with compact support, here we remove this assumption when
considering nonlocal Dirichlet constraints. We do reinstate this hypothesis when studying
the nonlocal Neumann problem. 
 
To prove the well-posedness of the equations, and thus obtain consistency of the numerical scheme,
 we follow the Galerkin approach.
One of the main difficulties that we face when pursuing this endeavor comes from 
the nonlocal operator. Unlike the Laplacian, our  nonlocal map, $K$,  does not define a bilinear form with domain
$H^1(\Omega) \times H^1(\Omega)$,
but with domain $L^2(\Omega) \times L^2(\Omega)$ instead.
This forces us to pick solutions in a space that does not have enough regularity. 
As a result, we are not able to use Aubin's compactness theorem to 
prove that the nonlinear terms in the equations converge weakly to the appropriate limit.
To get around this issue, we extend the nonlocal Gray-Scott system to include
complementary equations for two additional variables, $u_x$ and  $v_x$.
We then prove that at the level of the Galerkin approximation, these variables correspond to the
derivatives of the approximations $(u^m,v^m)$, and thus obtain the needed regularity to carry out the rest of the proof.
The result is the following theorem, which gives conditions on the parameters and the initial data 
that guarantee the existence of small-time weak solutions to the nonlocal Gray-Scott model.

\begin{Theorem}\label{t:existence}
Let $u_0,v_0,\partial_x u_0,\partial_x v_0$ be in $L^2(\Omega)$. Then,
there exists positive constants $C_1,C_2$ and  $T$, such that if
\[ \| u_0\|_{L^2(\Omega)} + \| v_0\|_{L^2(\Omega)} <C_1,\quad \| \partial_x u_0\|_{L^2(\Omega)} + \| \partial_x v_0\|_{L^2(\Omega)} <C_2,\]
 the system of equations
\begin{align} 
\langle u_t, \phi \rangle_{(V' \times V)} + d_u B_{D,N} [ u, \phi] & =  ( - uv^2 + f(1-u), \phi)_{L^2(\Omega)},\\
\langle v_t, \phi \rangle_{(V' \times V)} + d_v B_{D,N} [ v, \phi] & =  (  uv^2 - (f+\kappa)v), \phi)_{L^2(\Omega)},
\end{align}
has a unique weak solution
 $(u,v) \in [L^2(0,T, H^1(\Omega) )] \times [L^2(0,T, H^1(\Omega)] $ 
valid on the time interval $[0,T]$,
and satisfying $u(x,0) = u_0$ and $v(x,0)=v_0$.
\end{Theorem}

In the theorem the letter $V$ denotes the space of functions being considered, which can vary
depending on whether we are solving the Dirichlet problem, or the Neumann problem. 
In either case, the space $V$ is equivalent to $L^2(\Omega)$, as further explained in Section~\ref{s:preliminaries}. The bilinear form $B_{D,N}$ represents the action of the operator $K$, and it is also defined in Section~\ref{s:preliminaries}.
 Here, and throughout the paper, $\Omega$ always represents the domain of interest where the equations are
defined, while $\Omega_0$ denotes the region where the nonlocal boundary constraints are imposed. We then 
let 
$\tilde{\Omega} = \Omega \cup \Omega_0$ and refer to this region as the extended domain.

\begin{Remark}
A couple of remarks are in order.

\begin{enumerate}[1.]

\item Because nonlocal problems do not satisfy a particular type of `local' boundary condition on $\partial \Omega$,
 to construct our Galerkin approximations we use a special set of basis functions. This basis includes elements in $L^2(\Omega)$
 that satisfy homogenous Dirichlet boundary conditions, as well as elements that satisfy  Neumann and periodic boundary conditions on $\partial \Omega$.
Thus, this choice of basis functions allows for the existence of solutions that are discontinuous across the boundary,
so that the results presented here are consistent with previous work by Du and Yin \cite{du2019}. 

\item In the Neumann case, one can extend  basis elements from functions defined on $\Omega$ to functions defined on $\tilde{\Omega} = \Omega \cup \Omega_0$. This is done by
treating the Neumann constraints as an exterior problem, as is shown in 
\cite[Proposition 2.2]{burkovska2021}, (see also comments following Proposition \ref{p:burkovska} below and Lemma \ref{l:exteriorproblem}). 
This result also allows us to consider initial data, $(u_0,v_0)$, that is defined only on $\Omega$, even in the Neumann case, since reference
\cite[Proposition 2.2]{burkovska2021} shows that this boundary constraint provides enough information to determine the form of these functions on $\Omega_0$.

\item Finally, Theorem \ref{t:existence} shows that if we start initial data that is continuous on the computational domain (i.e. in $H^1(\Omega)$),  then we can expect the weak solution to also be continuous on this region.
However, we make no claims as to the regularity of the solution on $\Omega_0$. In particular, 
our result allows for solutions that are discontinuous across the boundary $\partial \Omega$.

\end{enumerate}
\end{Remark}

{\bf Outline:}
We finish this short introduction by establishing the notation that will be used throughout the paper and 
stating our main assumptions.
In Section~\ref{s:preliminaries} we provide a summary of preliminary results that are known in the literature.
These are then used in Section~\ref{s:existence} to prove the well-posedness of the nonlocal Gray-Scott model.
Section~\ref{sec:num} contains a detailed description of our numerical scheme along with 
the results of non-trivial numerical experiments that prove
the algorithm has first order convergence. In this section we also use our algorithm to explore the effects of nonlocal diffusion on the formation of single pulse solutions. We conclude the paper with a discussion and future directions in Section~\ref{s:discussion}.

\subsection{Model Equations}

We consider the nonlocal Gray-Scott model, 
\begin{align}\label{e:GS}
 u_t= & d_u \blue{K u} - uv^2 + f(1-u)\\
 \nonumber
 v_t= & d_v \blue{K v} + uv^2 -(f+\kappa)v
\end{align}
 posed on a bounded interval, $\Omega = [-L,L]$.
Here, the operator $K$ represents a nonlocal form of diffusion and is given by
\[ \blue{K u} = \int_{\tilde{\Omega}} (u(y) - u(x)) \gamma(x,y) \; dy,\]
where the domain of integration, $\Omega \subset \tilde{\Omega} \subseteq \R$, 
depends on the type of boundary constraints being considered.
In this paper we look at nonlocal Dirichlet and Neumann boundary constraints. 
We also restrict our attention to convolution kernels $\gamma$ satisfying 
Hypothesis \blue{\ref{h:kernelD} or Hypothesis \ref{h:kernelN}},  given bellow.  
As in the case of the `diffusive' Gray-Scott model, the constants $d_u,d_v$ are assumed to be small and represent the
rate of diffusion of the variables $u$ and $v$, respectively, while the coefficients $f,k$ are assumed to be constant.

\begin{Hypothesis}\label{h:kernelD}
The convolution kernel, $\gamma(z) = \gamma(|z|)$, is a positive radial function in $H^1(\R) \cap W^{1,1}(\R)$,
which is non-zero on all of $\R$ and satisfies:
 \begin{itemize}
 \item $\gamma (x,y) = \gamma(|x-y|) = \gamma (|z|)$, and
 \item $\int_\R z^2 \gamma(|z|) \;dz < \infty$.
 \end{itemize}
\end{Hypothesis}

Examples of kernels that meet Hypothesis \ref{h:kernelD} are:
 \begin{itemize}
 \item $\gamma(x) = \exp( -x^2)$,
 \item $ \gamma(x) = \exp(-|x|)$.
 \end{itemize}
 
\begin{Hypothesis}\label{h:kernelN}
The convolution kernel, $\gamma(z) = \gamma(|z|)$, is a positive radial function in $H^1(\R) \cap W^{1,1}(\R)$, with compact support  and satisfying: 
 \begin{itemize}
 \item $\gamma (x,y) = \gamma(|x-y|) = \gamma (|z|)$, and
 \item $\int_\R z^2 \gamma(|z|) \;dz < \infty$.
 \end{itemize}
\end{Hypothesis}

\subsection{ Notation}\label{sub:notation}
As already mentioned, we consider two types of boundary conditions,
\begin{enumerate}[I.]
\setlength \itemsep{2ex}
\item {\bf Nonlocal Dirichlet:} 
\begin{align} \label{e:D}
 u & = 0 \hspace{15ex} \mbox{for all} \quad x \in \R\setminus \Omega,\\
\nonumber
 \tilde{\Omega} & = \R.
\end{align}
In this case 
the operator $K$ is identified with a convolution kernel, $\gamma$,  strictly positive on all of $\R$ and satisfying Hypothesis \ref{h:kernelD}. 
We also assume the unknown variables $u,v$ belong to the space $V_D$, defined as
\[  V_D = \{ u \in L^2(\R) : u(x) = 0 \quad \forall x \in \Omega^c\}\]
and endowed with the $L^2(\Omega)$ norm.

\item {\bf Nonlocal Neumann:} 
\begin{align} \label{e:N} 
\blue{K u} & = 0 \hspace{15ex} \mbox{for  all} \quad x \in \Omega_0, \\
\nonumber
\tilde{\Omega} & = \Omega \cup \Omega_0.
\end{align}
In this case we assume the convolution kernel, $\gamma$, has compact support and satisfies Hypothesis \ref{h:kernelN}. 
 For simplicity, we assume $\gamma  = \chi_R \tilde{\gamma}$, where $ \tilde{\gamma}$ is a positive radial function in $W^{1,1}(\R) \cap H^1(\R)$,
 and $\chi_R = \chi_R(|x|) \in C^\infty_0(\R)$ is a cut off function with support in the interval $[-R,R]$. Consequently,
 $\Omega_0= [-L-R,-L] \cup [L,L+R]$.
 In addition, we assume that the unknowns $u,v$ belong to the space $V_N$, which we define as
\[ V_N = \{ u \in L^2(\tO): \blue{K u}(x) = 0 \quad x \in \Omega_0\}.\]
We equip this space with the norm
$ \| u\|^2_{V_N} = \| u\|^2_{L^2(\Omega)} + |u|^2_{V_N}$, where
\[  |u|^2_{V_N} = \frac{1}{2} \int_{\tO} \int_{\tO} (u(y) - u(x))^2 \gamma(x,y) \;dy \;dx. \]
In Section \ref{s:preliminaries} we will also refer to the closed subspace 
\[\tilde{V}_N = \{  u \in L^2(\tO): \blue{K u}_{|\Omega_0} = 0, \quad \int_{\tO} u = 0 \}.\]

 \end{enumerate}

\section{Preliminary Results}\label{s:preliminaries}

\subsection{Bilinear Forms}
In this section we consider the bilinear forms associated with the operator $K$, and summarize
some of their properties.

We start by defining the two bilinear forms we will be working with.
\begin{Definition}\label{d:BD}
We define the bilinear form $B_D: V_D \times V_D \longrightarrow \R$ by the expression
\[ B_D[u,\phi] = \frac{1}{2} \int_\R \int_\R (u(y) - u(x)) \gamma(x,y) (\phi(y) - \phi(x)) \;dy\;dx,\]
where the kernel $\gamma $ satisfies Hypothesis \ref{h:kernelD}.
\end{Definition}

\begin{Definition}\label{d:BN}
We define the bilinear $B_N: V_N \times V_N \longrightarrow \R$ by the expression
\[ B_N[u,\phi] = \frac{1}{2} \int_{\tO} \int_{\tO} (u(y) - u(x)) \gamma_R(x,y) (\phi(y) - \phi(x)) \;dy\;dx,\]
where $\gamma_R$  satisfies Hypothesis \ref{h:kernelN} and  thus has compact support.
\end{Definition}
Notice that thanks to our definition of the space $V_D$ and $V_N$ we have that
\begin{align*}
 B_D[u, \phi] = & - (\blue{K u}, \phi)_{L^2(\Omega)},\\
 B_N[u, \phi] = & - (\blue{K u}, \phi)_{L^2(\Omega)}.
\end{align*}

The following two lemmas show that the bilinear forms $B_D$ and $B_N$ are bounded and
coercive.

\begin{Lemma}
Let $B_D: V_D \times V_D \longrightarrow \R$ be defined as in \ref{d:BD}, then there exists positive constants 
$C$ and $\beta$, such that
\begin{enumerate}[i)]
\item $B_D[u, \phi ] \leq C \| u\|_{L^2(\Omega)} \| \phi\|_{L^2(\Omega)} $,
\item $B_D[u,u] \geq \beta \| u\|^2_{L^2(\Omega)}$.
\end{enumerate}
\end{Lemma} 
\begin{proof}
Because the kernel $\gamma(x,y) = \gamma(|x-y|)$ is in $L^2(\Omega \times \Omega)$, the convolution
with $\gamma$ is a Hilbert-Schmidt operator. As a result the map
$- K: L^2(\Omega) \longrightarrow L^2(\Omega)$, written as
$$ -K u =  \Gamma u - \gamma \ast u, \qquad \Gamma = \int_\R \gamma (z) \;dz,$$
is a compact perturbation of a constant multiple of the
 identity, and it is therefore bounded.
 Since $B_D[u, \phi] = - (\blue{K u}, \phi)_{L^2(\Omega)}$,  we then obtain the boundedness of the bilinear form using Cauchy-Schwartz.
 
Keeping in mind that the kernel satisfies $\gamma(x,y) >0 $ for all $ (x,y) \in \R^2$, and using the notation
 \[ -K u = \int_\R (u(x) -u(y) ) \gamma(x,y)\;dy\]
 one can prove that the only element in the null-space of the operator $K$ is the trivial solution. 
Indeed, if $Ku =0$ for some $u \in V_D$ then, $-(K \ast u, u)_{L^2(\R)} =0$, and since
 \[ -(K \ast u, u)_{L^2(\R)} = \int_\R \int_\R (u(x)-u(y) )^2 \gamma(x,y) \;dy \;dx =0\] 
one must have $(u(x) - u(y))=0$ for a.e. $(x,y) \in \R^2$. 
 It then follows that $u$ is a constant in $\R$ and since $u =0 $ on $\Omega^c$, we arrive at $u \equiv 0$.
The invertibility of the operator $-K$ then shows that the bilinear form $B_D$ is coercive. 
To see this, suppose to the contrary that $B_D$ is not coercive and thus,
there exists a sequence, $\{ u_n\}$, of non-trivial elements in $ V_D$ with norm $\|u_n \|_{L^2(\Omega)} =1$ such that $B_D[u_n, u_n] \rightarrow 0$. As shown below this implies that $K(u_n) \rightarrow 0$ in the $L^2$-norm. Since $-K$ is invertible we then obtain that
\[1= \|u_n\|_{L^2(\Omega)} \leq C \|K(u_n) \|_{L^2(\Omega)} \to 0,\]
reaching a contradiction.

To show $K(u_n) \rightarrow 0$,  notice that
\begin{align*} B_D[u,\phi] = &
 \frac{1}{2} \int_{\R \times \R} (u(y) - u(x)) \gamma(x,y) (\phi(y) - \phi(x)) \; dA\\
 & \leq   \frac{1}{2} \left( \int_{\R \times \R} (u(y) - u(x))^2 \gamma(x,y)  \; dA \right)^{1/2}
 \left( \int_{\R \times \R}  \gamma(x,y) (\phi(y) - \phi(x))^2 \; dA \right)^{1/2}\\
 & \leq \frac{1}{2} \Big( B_D[u,u]\Big)^{1/2}  \Big( B_D[\phi,\phi]\Big)^{1/2} 
 \end{align*}
 where the second line follows by an application of the Cauchy-Schwartz inequality.
Letting  $u = u_n$ and taking $\phi = -K(u_n)$, we see that
\begin{align*}
(K(u_n),K(u_n) )_{L^2(\Omega)} = & B_D[u_n,-K(u_n)] \\
\| K(u_n)\|^2_{L^2(\Omega)}& \leq  \frac{1}{2} \Big( B_D[u_n,u_n]\Big)^{1/2}  \Big( B_D[K(u_n),K(u_n)]\Big)^{1/2} \\
\| K(u_n)\|^2_{L^2(\Omega)}& \leq \frac{1}{2} \Big( B_D[u_n,u_n]\Big)^{1/2} \Big( C \| K(u_n)\|_{L^2(\Omega)}\| K(u_n)\|_{L^2(\Omega)} \Big)^{1/2}\\
\| K(u_n)\|_{L^2(\Omega)} & \leq \frac{\sqrt{C}}{2} \Big( B_D[u_n,u_n]\Big)^{1/2}, 
\end{align*}
where in the third line we used the boundedness of the bilinear form. The result then follows from our assumption that 
 $B_D[u_n,u_n] \rightarrow 0$ as $n\to \infty$.
 \end{proof}

 In Lemma \ref{l:BN} we prove that the bilinear form $B_N[u,v]$ is bounded and coercive. 
 To do this we employ the following proposition, which appears in \cite[Proposition 2.2]{burkovska2021}.
 
 \begin{Proposition}\label{p:burkovska}
 The space $(V_N, \| \cdot\|_{V_N})$ is a Hilbert space that is equivalent to $L^2(\Omega)$.
 \end{Proposition}

Heuristically, the results of the above proposition follow from the Neumann boundary constraint,
which defines an exterior problem:
\begin{align*} 
Ku =0 & \qquad x \in \Omega_0\\
u = g & \qquad x \in \Omega.
\end{align*}
 In \cite{burkovska2021} the authors prove that this nonlocal Dirichlet problem has a unique solution. Thus, the values of $u$ in the set $\Omega_0$ are determined by the values of $u$ in  the domain $\Omega$, leading to the inequality   $\| u\|_{V_N} \leq C \|u\|_{L^2(\Omega)}$.
As a result, the two norms, $\| \cdot \|_{V_N} $ and $\|\cdot\|_{L^2(\Omega)}$, are equivalent, and the result of Proposition \ref{p:burkovska} then follow.
For reference we provide a detailed definition, and a proof of the existence of solutions to the exterior problem in Subsection \ref{ss:basis}, Lemma \ref{l:exteriorproblem}

Notice that as a  consequence of Proposition \ref{p:burkovska} we also obtain that the embeddings 
$V_N \subset L^2(\Omega) \subset V'_N$ are dense, and therefore these spaces define 
a Gelfand triple. Here $V'_N$ denotes the dual to $V_N$.

\begin{Lemma}\label{l:BN}
Let $B_N: \tilde{V}_N \times \tilde{V}_N \longrightarrow \R$ be defined as in \ref{d:BN}, then there exists positive constants 
$C_1,C_2$ and  $\beta$, such that
\begin{enumerate}[i)]
\item $B_N[u, \phi ] \leq C_1 \| u\|_{L^2(\Omega)} \| \phi\|_{L^2(\Omega)} $,
 \item $B_N[u, \phi ] \leq C_2 \| u\|_{V_N} \| \phi\|_{V_N} $,
\item $B_N[u,u] \geq \beta \| u\|^2_{L^2(\Omega)}$.
\end{enumerate}
\end{Lemma} 
\begin{proof}[Sketch proof]
Using the fact that $B_N[u, \phi] = -(\blue{K u}, \phi)_{L^2(\Omega)}$, the assumption that $\gamma_R \in L^1(\R)$ 
together with Young's inequality for convolutions shows that
\[ B_N[u, \phi] \leq C \| u\|_{L^2(\Omega)} \|\phi \|_{L^2(\Omega)}, \]
 for some positive constant $C$. Since by Proposition \ref{p:burkovska} the norms $\| \cdot \|_{V_N}$ and $ \| \cdot \|_{L^2(\Omega)}$ are equivalent, we obtain item ii).

The third item in the Lemma follows from a result of Mengesha and Du  in \cite[Proposition 1]{mengesha2013} (or \cite{mengesha2014nonlocal}, Proposition 2) where it is shown that 
this type of bilinear form satisfies a generalized Poincar\'e inequality, provided we work with the space $\tilde{V}_N$ instead of $V_N$.
\end{proof}

\section{Existence of Solutions}\label{s:existence}

Our goal in this section is to prove the existence of solutions to the weak formulation 
of the nonlocal Gray-Scott model. That is, we want to show that there is a positive time, $T$, and a pair of solutions $u,v$, each in $L^2( [0,T], V) $ and with $u_t,v_t $ in $ L^2( [0,T], V') $, satisfying the initial conditions $u(x,0) =u_0, v(x,0)=v_0$, as well as the equations
\begin{align*}
\langle u_t, \phi \rangle_{(V' \times V)} + d_u B_{D,N} [ u, \phi] = & ( - uv^2 + f(1-u), \phi)_{L^2(\Omega)},\\
\langle v_t, \phi \rangle_{(V' \times V)} + d_v B_{D,N} [ v, \phi] = & (  uv^2 - (f+\kappa)v), \phi)_{L^2(\Omega)},
\end{align*}
for all $\phi \in V$ and for a.e. $t \in [0,T]$.

In terms of notation, the subscript $D$ or $N$ in the expression for the bilinear form denote
the type of boundary constraints that are being considered, nonlocal Dirichlet or Neumann, respectively.
The choice of space $V$ also depends on the boundary constraints. Using the notation established in 
Subsection \ref{sub:notation}, we let $V= V_D$ in the case of Dirichlet boundary constraints,
and pick $V= V_N$ in the case of Neumann boundary constraints.

To prove the existence of weak solutions we use the Galerkin method.
Before proceeding with the proof, we recall the steps involved in this approach in order to highlight 
some of the challenges inherent in our problem.
Very briefly this method consists in picking an orthonormal basis for the space $V$,
 and approximating the unknowns, $u$ and $v$, 
by a \emph{finite} linear combination of elements in this basis.
Inserting this approximation into the weak formulation of the problem results in 
a finite dimensional system of ordinary differential equations.
Because these differential equations define a continuous vector field,  we are guaranteed the existence of a solution for
some  time interval $[0,T]$ thanks to the Peano-Cauchy theorem.
One  then  shows that the sequences of approximations $\{u^m\}_{m=1}^\infty, \{v^m\}_{m=1}^\infty $ are bounded in $L^2([0,T), V)$,
so that weakly convergent subsequences can be extracted.
Using Aubin's theorem one then proves that these sequences converge strongly in an appropriate space, 
and that their limits, $\bar{u}$ and  $\bar{v}$, satisfy the weak formulation of the problem.

One of the main difficulties we face comes from the nonlinear terms in the equation.
Because we are working with bilinear forms that come from integral operators, 
we are only able to obtain energy estimates in the $L^2$ norm. 
Consequently, we can only bound the approximating sequences in $L^2(0,T,L^2(\Omega))$, 
 and thus we cannot use Aubin's compactness theorem to extract a strongly convergent subsequence. As a result, the 
 weak limit $(\bar{u},\bar{v})$ does not necessarily satisfy the weak formulation of our problem. 
 In particular, it is not guaranteed that the nonlinear terms converge.

 To get around this issue we consider the following augmented systems.  In the Dirichlet case we
look at
\begin{equation}\label{e:weakform}
\begin{array}{r c l}
\langle u_t, \phi \rangle_{(V' \times V)} + d_u B_{D} [ u, \phi] & = &  ( - uv^2 + f(1-u), \phi)_{L^2(\Omega)},\\
\langle v_t, \phi \rangle_{(V' \times V)} + d_v B_{D} [ v, \phi] & =  &(  uv^2 - (f+\kappa)v), \phi)_{L^2(\Omega)},\\[1ex]
 \langle (u_{x})_t, \phi \rangle_{(V' \times V)} + d_u \Gamma ( u_x, \phi)_{L^2(\Omega)} & =&  d_u( J_D u, \phi)_{L^2(\Omega)}  +  ( - u_x v^2 - 2uvv_x - fu_x, \phi)_{L^2(\Omega)},\\
\langle (v_x)_t, \phi \rangle_{(V'\times V)} + d_v \Gamma(v_x, \phi)_{L^2(\Omega)} & =& d_v( J_D v, \phi)_{L^2(\Omega)} + (  u_xv^2 +2uvv_x - (f+\kappa)v_x, \phi)_{L^2(\Omega)},
\end{array}
\end{equation}
while in the Neumann case we study
\begin{equation}\label{e:weakformN}
\begin{array}{r c l}
\langle u_t, \phi \rangle_{(V' \times V)} + d_u B_{N} [ u, \phi] & = &  ( - uv^2 + f(1-u), \phi)_{L^2(\Omega)},\\
\langle v_t, \phi \rangle_{(V' \times V)} + d_v B_{N} [ v, \phi] & =  &(  uv^2 - (f+\kappa)v), \phi)_{L^2(\Omega)},\\[1ex]
\langle (u_{x})_t, \phi \rangle_{(V' \times V)} + d_u \Gamma (u_x, \phi)_{L^2(\Omega)} & =& d_u (J_N u,\phi)_{L^2(\Omega)} +  ( - u_x v^2 - 2uvv_x - fu_x, \phi)_{L^2(\Omega)},\\
\langle (v_x)_t, \phi \rangle_{(V'\times V)} + d_v \Gamma (v_x , \phi)_{L^2(\Omega)} & =& d_v (J_N v,\phi)_{L^2(\Omega)} + (  u_xv^2 +2uvv_x - (f+\kappa)v_x, \phi)_{L^2(\Omega)}.
\end{array}
\end{equation}
In this last set of equations the constant $\Gamma$ is defined as $\Gamma= \int_\R \gamma(|x|) \;dx$ in the Dirichlet case, or as $\Gamma= \int_\R \gamma_R(|x|) \;dx$ in Neumann case.
We also define the operators $J_{D,N}:V_{D,N} \longrightarrow L^2(\R)$ by
\[J_D u = \int_{\R} (u(y) - u(x) ) \partial_x \gamma(x,y) \;dy, \qquad J_N u = \int_{\tilde{\Omega}} (u(y) - u(x) ) \partial_x \gamma_R(x,y) \;dy.\]
Notice that by Hypotheses \ref{h:kernelN} and \ref{h:kernelD} we have that $\partial_x \gamma$ and $\partial_x \gamma_R$ are in $L^2(\R) \cap L^1(\R)$, so that the operators $J_{D,N}$ are well defined.
 In the analysis, we also impose the initial conditions,
 \[u(x,0) = u_0(x), \quad v(x,0) = v_0(x), \quad u_x(x,0) = \partial_x u_0(x), \quad v_x(x,0) = \partial_x v_0(x),\]
  for both problems, and to start we also assume that the unknowns $u_x,v_x$ are independent from $u,v$.

Before continuing, let us motivate our choice of auxiliary equations. In both cases, our starting point is the original 
nonlocal Gray-Scott model \eqref{e:gs_nonlocal}. Assuming for the moment that $u$ and $v$ are smooth, we can take a partial derivative with respect to $x$ of these equations and find,
\begin{align*}
\partial_t (\partial_x u)= & d_u \partial_x( K  u) - v^2 \partial_x u- 2 uv \partial_x v - f \partial_x u, \\
\partial_t( \partial_x v)= & d_v \partial_x (K  v) + v^2 \partial_x u + 2uv \partial_x v -(f+\kappa)\partial_x v.
 \end{align*}
 In the Dirichlet case, the last two equations in the augmented system \eqref{e:weakform} come from considering the weak formulation (in $L^2(\Omega)$) of the above expressions and
 the following calculation,
 \begin{align*}
 \partial_x(Ku) = 
 & \int_{\R} (u(y) - u(x) ) \partial_x \gamma(x,y) \;dy  - \partial_x u \int_{\R} \gamma(x,y) \;dy\\
 &=  J_Du   -  \Gamma \partial_x u.
 \end{align*}

On the other hand, since in the Neumann case the extended domain, $\tilde{\Omega}$, is a bounded subset of $\R$, we find that
\begin{align*}
 \partial_x(Ku) = 
 & \int_{\tilde{\Omega}} (u(y) - u(x) ) \partial_x \gamma_R(|x-y|) \;dy  - \partial_x u \int_{\tilde{\Omega}} \gamma_R(|x-y|) \;dy\\
 &=   J_N u   - \partial_x u \Gamma(x)
 \end{align*}
where $\Gamma(x) = \int_{\tilde{\Omega}} \gamma_R(|x-y|) \;dy$. Using this identity and the fact that $\Gamma(x) = \Gamma$, a constant, for $x \in \Omega$, we arrive at the last two equations of the weak formulation \eqref{e:weakformN}, shown above.

The remainder of this section is organized as follows.
In Subsection \ref{ss:basis} we define a special class of basis function, 
which we will then use to prove existence of weak solutions.
In Subsection \ref{ss:approx} we show that the Galerkin approximations, $u^m,v^m,u^m_x,v^m_x$, to systems \eqref{e:weakform} and \eqref{e:weakformN}
satisfy $\partial_xu^m=u^m_x , \partial_xv^m=v^m_x$ in $\Omega$.  This result will then help us prove that the sequences 
of approximating solutions we construct,
$\{ u^m\}, \{v^m\}$, are bounded in $L^2([0,T],H^1(\Omega))$, so that we may
apply Aubin's theorem. We perform this analysis in Subsection \ref{ss:bounds}, and then prove our main existence theorem
in Subsection \ref{ss:main}.

\subsection{The basis elements}\label{ss:basis}
To construct our sequence of approximate solutions, we  let $\{ \phi_n\}$ denote a basis for $V$
with the particular property that its elements do not satisfy any specific type 
of {\it local} boundary condition. 

For example, given a domain $\Omega = [ -L,L]$ and assuming nonlocal Dirichlet boundary constraints, we can pick as a basis for $V = V_D$ 
the set of functions 
\begin{equation}\label{e:basis}
\{ \phi_n\} = \{\frac{1}{2L}\} \cup \left\{\frac{1}{L} \cos\left( \frac{(2n-1) \pi x}{2 L} \right),  \;\frac{1}{L} \sin \left( \frac{(2n-1) \pi x}{2 L} \right) \right\},
\end{equation}
 with
 $n \in \N $.
One can check that this is an orthonormal set. To see why these functions
 form a basis for $V_D $, notice first that the full set
$\{\psi_m\} = \{1\} \cup \{ \cos( \frac{m \pi x}{2 L}),  \; \sin ( \frac{m \pi x}{2 L})\}$, with $m \in \N$, forms a basis for
the space $X = L^2( [-2L,2L))$ with periodic boundary conditions. 
Since $V_D$ is a subspace of $X$, it follows that these last functions can describe elements in $V_D$. However,
when the functions $\psi_m$ are restricted to the domain $[-L,L]$, one finds that some of them can be written
in terms of other elements in the set $\{\psi_m\}.$ 
In particular, using the inner product in $V_D \sim L^2([-L,L])$, one finds that
\[  \int_{-L}^L  \cos \left( \frac{m \pi x}{2 L} \right) ^2 \;dx = L, \quad  \int_{-L}^L  \sin \left( \frac{m \pi x}{2 L} \right) ^2 \;dx = L, \quad   \int_{-L}^L  \cos \left( \frac{n \pi x}{2 L} \right) \sin \left( \frac{m \pi x}{2 L} \right)  \;dx = 0,\]
while
\begin{align*}
  \int_{-L}^L  \cos \left( \frac{n \pi x}{2 L} \right) \cos \left( \frac{m \pi x}{2 L} \right)   \;dx  &= 
\frac{2L}{\pi} \left( \frac{\sin \left(\frac{ \pi}{2}  (m-n) \right) }{m-n} +\frac{\sin \left(\frac{ \pi}{2}  (m+n) \right) }{m+n}  \right),\\
  \int_{-L}^L  \sin \left( \frac{n \pi x}{2 L} \right) \sin \left( \frac{m \pi x}{2 L} \right)   \;dx  &= 
\frac{2L}{\pi} \left( \frac{\sin \left( \frac{ \pi}{2}  (m-n) \right) }{m-n} -\frac{\sin \left(\frac{ \pi}{2}  (m+n) \right) }{m+n}  \right).
\end{align*}
These last integrals are zero if $n$ and $m$ are both even or odd, but are non-zero if, for example, $m$ is even and $n$ is odd.
This implies that we can write the even terms, $ \cos( \frac{2 n \pi x}{2 L})$ and $ \sin ( \frac{2n \pi x}{2 L})$, as
a linear combination of  the odd terms, $ \cos( \frac{(2 n-1) \pi x}{2 L})$ and $\sin ( \frac{(2n-1) \pi x}{2 L})$, respectively.
We can therefore discard either all even terms or all odd terms and still have a complete orthogonal basis for $V_D$.
 We choose to discard the even terms, since by keeping the odd terms we are not imposing any kind of local boundary condition. Indeed, the elements in the set $\{\phi_n\}$  do not all satisfy periodic, homogeneous Dirichlet, or homogeneous Neumann boundary conditions on $[-L,L]$.

 In the case of Neumann boundary constraints, where  the space that we use is $V = V_N \subset L^2(\tO)$,
we can again consider the same set of basis functions, $\{ \phi_n\}$, as a starting point. 
One can then use the results from \cite{burkovska2021}, 
where the Neumann boundary constraint is viewed as an exterior problem,
to extend the elements in this set to functions that are defined in all of  $\tO = \Omega \cup \Omega_0$. 
These new functions, $\{ \tilde{\phi}_n\}$, then form an orthonormal basis
 (with respect to the $L^2(\Omega)$ inner product) for the space $V_N$.

To validate the above statements, we summarize the results from \cite{burkovska2021} in Lemma \ref{l:exteriorproblem}, and provide a sketch of the proof.
For this, we first need to set up the correct notation. In what follows we let $g \in L^2(\Omega)$ and consider the exterior problem associated with the Neumann constraints,
\begin{align}\label{e:exteriorp2}
K u = 0 & \qquad x \in \Omega_0\\
\nonumber
u = g  & \qquad x \in \Omega.
\end{align}
We also define the space $$V_N^0 = \{ u \in L^2(\Omega \cup \Omega_0) : B_N[u,u]< \infty, u(x) = 0 \quad \mbox{for} \quad x\in \Omega\},$$
with norm given by $\|  u \|_{V^0_N}  = \| u\|_{L^2(\Omega_0)} + |u|_{V^0_N} $, where $ |u|^2_{V^0_N}= B_N[u,u]$.
The weak formulation of the exterior problem \eqref{e:exteriorp2} then takes the form,
\begin{equation}\label{e:weakexterior}
 B_N[\hat{u},v] = ( K \tilde{g}, v)_{L^2(\Omega_0)} \qquad \forall v \in V^0_N.
 \end{equation}
where $\hat{u} = u-\tilde{g}$, and $\tilde{g}$ is the extension of $g$ by zero to $\tilde{\Omega}$.

\begin{Lemma}\label{l:exteriorproblem}
Let $g \in L^2(\Omega)$ and define $\tilde{g}$ as its extension by zero to $\tilde{\Omega}$. 
Then, the nonlocal Dirichlet problem \eqref{e:exteriorp2}
has a unique weak solution, $u \in L^2(\tilde{\Omega})$, given by $u = \hat{u} +\tilde{g}$, where $\hat{u} \in V^0_N$  satisfies
\eqref{e:weakexterior}. Moreover, we have the following inequality,
\[ \| u\|_{L^2(\Omega_0)} \leq C \| g\|_{L^2(\Omega)}\]
which holds for some $C = C(\gamma) >0$.
\end{Lemma}
\begin{proof}
We consider the same bilinear form, $B_N$, from Definition \ref{d:BN}, except that now we view it as
an operator $B_N: V_N^0 \times V_N^0 \longrightarrow \R$. 
One can show that this bilinear form is continuous and bounded. Indeed, we have that
\begin{align*}
B_N[u,v] &= (-Ku,v)_{L^2(\Omega)} \\
& \leq \|Ku\|_{L^2(\tilde{\Omega})} \|v\|_{L^2(\tilde{\Omega})}\\
& \leq 2\| \gamma\|_{L^1(\R)} \|u\|_{L^2(\tilde{\Omega})} \|v\|_{L^2(\tilde{\Omega})}\\
& \leq 2\| \gamma\|_{L^1(\R)} \|u\|_{L^2(\Omega_0)} \|v\|_{L^2(\Omega_0)}\\
& \leq 2\| \gamma\|_{L^1(\R)} \|u\|_{V^0_N} \|v\|_{V^0_N}
\end{align*}
where the third line follows from extending $u$ by zero to $\R$, 
the relation $|Ku| \leq | \gamma \ast u| +|u| \| \gamma \|_{L^1(\R)}$, 
and Young's inequality for convolutions.

In addition, Proposition 1 in \cite{mengesha2013} shows that
there exists a constant $C>0$, such that  $\| u\|^2_{L^2(\tilde{\Omega})} \leq C B_N[u,u]$.
It then follows that the bilinear form is coercive, i.e.
\[ \|u\|^2_{V^0_N} \leq 2C B_N[u,u].\]
Therefore, by the Lax-Milgram theorem the weak formulation \eqref{e:weakexterior} has a unique solution $\hat{u}$, satisfying
$ \|\hat{u} \|_{V_N^0} \leq \| K \tilde{g} \|_{L^2(\Omega_0)} $. 
The inequality stated in the lemma then follows from the definition of $u$, i.e. $u = \hat{u} + \tilde{g}$, and
the following relations:
\[ \|u\|_{L^2(\Omega_0)} = \|\hat{u} \|_{L^2(\Omega_0)} \leq \|u\|_{V^0_N} \leq \| K \tilde{g} \|_{L^2(\Omega_0)}
=   \| K \tilde{g} \|_{L^2(\R)} \leq 2 \| \gamma\|_{L^1(\R)} \| g\|_{L^2(\Omega)}. \]
As before, the last two inequalities result from extending $g$ by zero to all of $\R$, the definition of the 
nonlocal operator $K$, and Young's inequality.
\end{proof}

\begin{Remark}
Suppose that $g \in H^1(\Omega)$ and that $u$ and $v$ solve the exterior problem \eqref{e:exteriorp2} with $u = g$ and $v = \partial_x g $ in $\Omega$.
Then, while it is clear that $v = \partial_x u$ for all $ x \in \Omega$, it is not necessarily true that $v = \partial_x u$ in the exterior domain, $\Omega_0$, since in general we do not have that $\partial_x(Ku) = K \partial_x u=0$ for $x \in \Omega_0$.
\end{Remark}

\subsection{An auxiliary system of equations}\label{ss:approx}

To show that the Galerkin approximations, $u^m,v^m,u^m_x,v^m_x$, to systems \eqref{e:weakform} and \eqref{e:weakformN}
satisfy $\partial_xu^m=u^m_x , \partial_xv^m=v^m_x$ in $\Omega$, 
in this subsection we consider the following auxiliary systems.

In the Dirichlet case, we use
\begin{equation}\label{e:weakaux}
\begin{array}{r c l}
\langle u_t, \phi \rangle_{(V' \times V)} + d_u B_{D} [ u, \phi] & = &  ( - \sigma(u)\sigma(v)^2 + f(1-u), \phi)_{L^2(\Omega)},\\
\langle v_t, \phi \rangle_{(V' \times V)} + d_v B_{D} [ v, \phi] & =  &(  \sigma(u)\sigma(v)^2 - (f+\kappa)v, \phi)_{L^2(\Omega)},\\[2ex]
\langle (u_x)_t, \phi \rangle_{(V' \times V)} + d_u \Gamma ( u_x, \phi)_{L^2(\Omega)} & = &  d_u (J_D u, \phi)_{L^2(\Omega)}   - (\sigma'(u)\sigma(v)^2 u_x, \phi)_{L^2(\Omega)}\\
& &  -2 ( \sigma(u) \sigma(v)\sigma'(v) v_x, \phi)_{L^2(\Omega)} -( f u_x, \phi)_{L^2(\Omega)},\\[1ex]
\langle (v_x)_t, \phi \rangle_{(V' \times V)} + d_v \Gamma ( v_x, \phi)_{L^2(\Omega)} & = & d_v (J_D u, \phi)_{L^2(\Omega)}+  (\sigma'(u)\sigma(v)^2 u_x, \phi)_{L^2(\Omega)} \\
& & + 2 ( \sigma(u) \sigma(v)\sigma'(v) v_x, \phi)_{L^2(\Omega)} -( (f+\kappa) v_x, \phi)_{L^2(\Omega)},
\end{array}
\end{equation}
while in the Neumann case we look at
\begin{equation}\label{e:weakauxN}
\begin{array}{r c l}
\langle u_t, \phi \rangle_{(V' \times V)} + d_u B_{N} [ u, \phi] & = &  ( - \sigma(u)\sigma(v)^2 + f(1-u), \phi)_{L^2(\Omega)},\\
\langle v_t, \phi \rangle_{(V' \times V)} + d_v B_{N} [ v, \phi] & =  &(  \sigma(u)\sigma(v)^2 - (f+\kappa)v, \phi)_{L^2(\Omega)},\\[2ex]
\langle (u_x)_t, \phi \rangle_{(V' \times V)} + d_u \Gamma ( u_x, \phi)_{L^2(\Omega)} & = & d_u (J_N u, \phi)_{L^2(\Omega)}   - (\sigma'(u)\sigma(v)^2 u_x, \phi)_{L^2(\Omega)} \\
& & -2 ( \sigma(u) \sigma(v)\sigma'(v) v_x, \phi)_{L^2(\Omega)} -( f u_x, \phi)_{L^2(\Omega)},\\[1ex]
\langle (v_x)_t, \phi \rangle_{(V' \times V)} + d_v  \Gamma ( v_x, \phi)_{L^2(\Omega)} & = & d_v (J_N u, \phi)_{L^2(\Omega)}  +  (\sigma'(u)\sigma(v)^2 u_x, \phi)_{L^2(\Omega)} \\
& &  + 2 ( \sigma(u) \sigma(v)\sigma'(v) v_x, \phi)_{L^2(\Omega)} -( (f+\kappa) v_x, \phi)_{L^2(\Omega)}.
\end{array}
\end{equation}

In both cases we consider the same initial conditions as before. We also define the  function $\sigma \in C^\infty(\R)$  by
\begin{equation}\label{e:sigma}
\sigma(z) = \left \{
\begin{array}{c c}
-2M & \mbox{if} \quad z< -2M\\
z & \mbox{if} \quad |z|\leq M\\
2M & \mbox{if} \quad z>2 M,
\end{array}
\right.
\end{equation}
for some $M>0$, with $\sigma$ 
satisfying $0 \leq \sigma'(z) \leq 1$ and $|\sigma(z)| \leq 2M$  for all $z\in \R$.

Notice that if both, $u,v$, are bounded in absolute value by a constant $D<M$, the systems shown above reduce to equations~\eqref{e:weakform} \blue{ and \eqref{e:weakformN}, respectively.}
Thus, our goal is to show that this condition is satisfied provided the $L^2(\Omega)$ norm of 
the initial functions, $u_0, v_0, \partial_x u_0, \partial_x v_0$, is small enough.
We then use this result to prove Proposition \ref{p:infinitybound} stated at the end of this subsection, where we show that  the Galerkin approximations $u^m,v^m,u^m_x,v^m_x$ to systems \eqref{e:weakform} and \eqref{e:weakformN} (or equivalently systems \eqref{e:weakfinite1}  and \eqref{e:weakfinite1N}, shown below, with $\sigma $ equal to the identity map) satisfy $\partial_x u^m = u^m_x$ and $\partial_x v^m= v^m_x$ on the domain $\Omega$.

We start by defining our Galerkin approximations for the variables $u,v, u_x,v_x$.
Given a fixed integer $m$, we select a set of $2m+1$ basis elements $\{\phi_k\}_{k=0}^{2m}$, with $\phi_0 = 1$ and with the property 
that $\partial_x \phi_k$ is in the span $\{ \phi_k\}_{k=0}^{2m}$, whenever $\phi_k$ is in this set. We then define
\begin{equation}\label{e:galerkin}
u^m =  \sum_{k=0}^{2m} d_k(t) \phi_k(x),\quad
 v^m =  \sum_{k=0}^{2m} e_k(t) \phi_k(x),\qquad
u^m_x =  \sum_{k=0}^{2m} \tilde{d}_k(t) \phi_k(x),\quad
 v^m_x=  \sum_{k=0}^{2m} \tilde{e}_k(t) \phi_k(x).
\end{equation}
satisfying
\[ u^m(x,0) = u_0(x), \quad v^m(x,0) = v_0(x), \quad u_x^m(x,0) = \partial_x u_0(x), \quad v^m_x(x,0) = \partial_x v_0(x).\]

\begin{Remark}\label{r:basis}
Notice that because of how the basis functions are defined, if say $\phi_j = \sin( \frac{(2j-1) \pi x}{2L} ) $ is part of the finite set $\{ \phi_k\}_{k=0}^{2m}$, then $\partial_x \phi_j = \frac{(2j-1) \pi }{2L} \cos( \frac{(2j-1) \pi x}{2L})  $ is not only in the span of this set, but is actually a constant multiple of an element in this same set.
\end{Remark}

Because in both cases, $V= V_D$ and $V= V_N$, the spaces $V \subset L^2(\Omega) \subset V'$ 
define a Gelfand triple, when the approximations \eqref{e:galerkin} are plugged back into systems \eqref{e:weakaux} and
\eqref{e:weakauxN} 
(with $\phi = \phi_k$) , one obtains an equivalent formulation for both problems,
\begin{equation}\label{e:weakfinite1}
\begin{array}{r c l}
( u^m_t, \phi_k )_{L^2(\Omega)} + d_u B_{D} [ u^m, \phi_k] & = &  ( - \sigma(u^m)\sigma(v^m)^2 + f(1-u^m), \phi_k)_{L^2(\Omega)},\\
(v^m_t, \phi_k )_{L^2(\Omega)} + d_v B_{D} [ v^m, \phi_k] & =  &(  \sigma(u^m)\sigma(v^m)^2 - (f+\kappa)v^m, \phi_k)_{L^2(\Omega)},\\[2ex]
( (u^m_x)_t, \phi_k )_{L^2(\Omega)}  + d_u \Gamma ( u^m_x, \phi_k)_{L^2(\Omega)} & = &  d_u (J_D u^m, \phi_k)_{L^2(\Omega)}  - (\sigma'(u^m)\sigma(v^m)^2 u^m_x, \phi_k)_{L^2(\Omega)} \\
& & -2 ( \sigma(u^m) \sigma(v^m)\sigma'(v^m) v^m_x, \phi_k)_{L^2(\Omega)}\\
& & -( f u^m_x, \phi_k)_{L^2(\Omega)},\\[1ex]
( (v^m_x)_t, \phi_k )_{L^2(\Omega)}+ d_v \Gamma ( v^m_x, \phi_k)_{L^2(\Omega)} & = & d_v (J_D  v^m, \phi_k)_{L^2(\Omega)} + (\sigma'(u^m)\sigma(v^m)^2 u^m_x, \phi_k)_{L^2(\Omega)}\\
& &  + 2 ( \sigma(u^m) \sigma(v^m)\sigma'(v^m) v^m_x, \phi_k)_{L^2(\Omega)} \\
& &-( (f+\kappa) v^m_x, \phi_k)_{L^2(\Omega)},
\end{array}
\end{equation}
\begin{equation}\label{e:weakfinite1N}
\begin{array}{r c l}
( u^m_t, \phi_k )_{L^2(\Omega)} + d_u B_{N} [ u^m, \phi_k] & = &  ( - \sigma(u^m)\sigma(v^m)^2 + f(1-u^m), \phi_k)_{L^2(\Omega)},\\
(v^m_t, \phi_k )_{L^2(\Omega)} + d_v B_{N} [ v^m, \phi_k] & =  &(  \sigma(u^m)\sigma(v^m)^2 - (f+\kappa)v^m, \phi_k)_{L^2(\Omega)},\\[2ex]
( (u^m_x)_t, \phi_k )_{L^2(\Omega)}  + d_u \Gamma( u^m_x, \phi_k)_{L^2(\Omega)} & = & d_u (J_N u^m, \phi_k)_{L^2(\Omega)}   - (\sigma'(u^m)\sigma(v^m)^2 u^m_x, \phi_k)_{L^2(\Omega)} \\
& & -2 ( \sigma(u^m) \sigma(v^m)\sigma'(v^m) v^m_x, \phi_k)_{L^2(\Omega)}\\
& & -( f u^m_x, \phi_k)_{L^2(\Omega)},\\[1ex]
( (v^m_x)_t, \phi_k )_{L^2(\Omega)}+ d_v \Gamma( v^m_x, \phi_k)_{L^2(\Omega)} & = & d_v (J_N  v^m, \phi_k)_{L^2(\Omega)} +  (\sigma'(u^m)\sigma(v^m)^2 u^m_x, \phi_k)_{L^2(\Omega)}\\
& &  + 2 ( \sigma(u^m) \sigma(v^m)\sigma'(v^m) v^m_x, \phi_k)_{L^2(\Omega)} \\
& &-( (f+\kappa) v^m_x, \phi_k)_{L^2(\Omega)}.
\end{array}
\end{equation}

Recall that in the Neumann case,  the basis elements in the Galerkin approximation,  $\tilde{\phi}$, consist of functions $\phi \in L^2(\Omega)$, 
which are extended to $\tilde{\Omega}$ by treating the Neumann constraint as an exterior problem (see \cite{burkovska2021} or Lemma \ref{l:exteriorproblem}). Thus, in this case, the functions $u^m,v^m,u^m_x,v^m_x$, given by \eqref{e:galerkin} are also defined on $\tilde{\Omega}$, so that the terms 
$B_N[u^m, \phi_k], B_N[v^m, \phi_k], (J_N u^m_x, \phi_k)_{L^2(\Omega)},(J_N v^m_x, \phi_k)_{L^2(\Omega)} $,
 in system \eqref{e:weakfinite1N} are all well defined.

Since $\{ \phi_n\}$ is an orthonormal basis, these two sets of equations
 can each be written as a system of nonlinear ordinary differential equations,
\[\begin{array}{r c l}
 I_{kj} d'_j + d_u M_{kj} d_j & =& F^{(1)}_k(d_j, e_j),\\
 I_{kj} e'_j + d_v M_{kj} e_j & = & F^{(2)}_k(d_j, e_j),\\
 I_{kj} \tilde{d}'_j + d_u \Gamma I_{kj} \tilde{d}_j & = &G^{(1)}_k(d_j, e_j, \tilde{d}_j, \tilde{e}_j),\\
 I_{kj} \tilde{e}'_j + d_v \Gamma I_{kj} \tilde{e}_j & = & G^{(2)}_k(d_j, e_j, \tilde{d}_j, \tilde{e}_j),
 \end{array}\]
where we have used  the convention that repeated indices represent summation in that index.
We also have that $I_{kj} = ( \phi_k , \phi_j)_{L^2(\Omega)}$ is the identity matrix,  $M_{kj} = B_D[\phi_k, \phi_j]$ in the Dirichlet case while $M_{kj} =  B_N[\phi_k, \phi_j]$ in the Neumann case, and $F^{(i)}, G^{(i)}$ denote 
the nonlinear terms of each equation, so that $F^{(i)}_k, G^{(i)}_k$ then represent the projections
of the nonlinearities onto the span of $\phi_k$. 

Since all nonlinearities are continuous, this system
 defines a Lipschitz continuous vector field in the variables $d_k,e_k,\tilde{d}_k,\tilde{e}_k$,
with a Lipschitz constant that depends on the function $\sigma$, and in particular on the upper bound $M$ (see equation \eqref{e:sigma}). However, notice that this Lipschitz constant is independent of the number of approximating functions used in \eqref{e:galerkin}.
As a result, we are able to apply the
Picard-Lindel\"of theorem and obtain the existence of
a constant $T=T(M)>0$, which is independent of $m$,  and a unique solution vector 
$\langle d_k(t),e_k(t), \tilde{d}_k(t), \tilde{e}_k(t) \rangle$, $k = 1,2,\cdots, m$,
which exists on the time interval $[0,T]$. In addition these functions live in the space $C^1[0,T]$.

The above discussion also allows to prove the next lemma.
\begin{Lemma}\label{l:derivatives}
Consider the system of ordinary differential equations defined either by \eqref{e:weakfinite1} or \eqref{e:weakfinite1N},  together with its solution vector 
$\langle d_k(t),e_k(t), \tilde{d}_k(t), \tilde{e}_k(t) \rangle$, $k = 1,2,\cdots, m$.
 When this solution vector is inserted into the Galerkin approximation \eqref{e:galerkin}, the functions $(u^m,v^m, u^m_x, v^m_x)$ satisfy,
\[ u^m_x = \partial_xu^m \quad \mbox{and} \quad v^m_x = \partial_x v^m\]
on the domain $\Omega$.
\end{Lemma}

\begin{proof}
We prove the results of the Lemma for the Dirichlet case, since the analysis for the Neumann case is very similar.
Suppose then that $u^m, v^m$ are of the form \eqref{e:galerkin} and that they solve the first set of equations in \eqref{e:weakfinite1}. 
Because the coefficients $d_k(t), e_k(t)$ are $C^1$ and the basis functions $\phi_k$ are in $C^\infty(\Omega)$ (since they are composed of sines and cosine functions), we then have that the derivatives, $\partial_x u^m$ and $\partial_x v^m$, are also elements in $V_D$. Plugging these functions in place of $u^m_x, v^m_x$ on
the second set of equations in system \eqref{e:weakfinite1}, we arrive at
\begin{equation}\label{e:modweak_or}
\begin{array}{r cl}
( (\partial_x u^m)_t, \phi_k )_{L^2(\Omega)}  + d_u \Gamma ( \partial_x u^m, \phi_k)_{L^2(\Omega)}  -  
d_u (J_D u^m, \phi_k)_{L^2(\Omega)} - ( G_\sigma^{(1)}( \partial_x u^m,\partial_x v^m), \phi_k) & = &  0,\\[1ex]
( (\partial_x v^m)_t, \phi_k )_{L^2(\Omega)}+ d_v \Gamma ( \partial_x v^m, \phi_k)_{L^2(\Omega)}   
- d_v (J_D  v^m, \phi_k)_{L^2(\Omega)} - ( G_\sigma^{(2)}(\partial_x u^m, \partial_x v^m), \phi_k) & = & 0,
\end{array}
\end{equation}
where we use $G_\sigma^{(i)}$ to denote the nonlinear maps that appear on the right hand side of \eqref{e:weakfinite1}.
Our goal is to show that the above equalities actually hold.

Because the identity $\partial_x(Ku) = J_D u - \Gamma \partial_x u$ is valid for all  those $u \in V_D$ which have derivatives, $\partial_x u $, also in $V_D$, 
we may write the left hand side of \eqref{e:modweak_or} as,
\begin{equation}\label{e:modweak}
\begin{array}{r cl}
( \partial_x(\partial_t u^m), \phi_k )_{L^2(\Omega)} - du ( \partial_x (K u^m), \phi_k)_{L^2(\Omega)}  - 
( \partial_x( F_\sigma^{(1)} (u^m,v^m) ), \phi_k)_{L^2(\Omega)},  \\[1ex]
( \partial_x(\partial_t v^m), \phi_k )_{L^2(\Omega)} - dv ( \partial_x (K v^m), \phi_k)_{L^2(\Omega)} - 
( \partial_x( F_\sigma^{(2)} (u^m,v^m) ), \phi_k)_{L^2(\Omega)}.
\end{array}
\end{equation}
Here, $F^{(i)}_\sigma(u,v)$ denotes the nonlinear maps appearing on the right hand side of the first two equations in \eqref{e:weakfinite1}.
The expressions in \eqref{e:modweak} follow from the
 the identity $\partial_x F^{(i)}_\sigma( u,v) = G^{(i)}_\sigma( \partial_x u, \partial_x v)$, which is valid for elements $u$ in $C^1(\Omega)$. 

Next, notice that letting
\begin{align*}
 N_1 := & \partial_t u^m - d_u  (K u^m) - F_\sigma^{(1)}(u^m,v^m),\\
 N_2 := & \partial_t v^m - d_v  (K v^m) - F_\sigma^{(2)}(u^m,v^m),
 \end{align*}
 the expressions \eqref{e:modweak} can be written as
 \[ ( \partial_x N_1, \phi_k)_{L^2(\Omega)},\]
 \[ ( \partial_x N_2, \phi_k)_{L^2(\Omega)}.\]
 Given the identity $B_D[u, \phi] = -(Ku, \phi)_{L^2(\Omega)}$, which holds for all $u \in V_D$, one recognizes that $( N_i, \phi_k)_{L^2(\Omega)} $, with $i =1,2$, are just the first two equations in system \eqref{e:weakfinite1}. Because $u^m$ and $v^m$ solve these two equations,  the projection of $N_i$ onto the span of $\{ \phi_k\}_{k=0}^{2m}$ is equal to zero. At the same time, since due to our choice of basis functions (see \eqref{e:galerkin} and Remark \ref{r:basis}), $\partial_x \phi_k$ is also a basis element whenever $\phi_k$ is in this set, it then follows that the projection of $\partial_x N_i$ onto the span of $\{ \phi_k\}_{k=0}^{2m}$ is also zero.
 As a result we obtain that $( \partial_x N_i, \phi_k)_{L^2(\Omega)} =0$ for $i =1,2$ and $k =1,2, \cdots, m$, and consequently the equations \eqref{e:modweak_or} are indeed valid.

Because the system of ordinary differential equations defined by the weak formulation of both, the Dirichlet and Neumann problem, has a unique solution, we conclude that $u^m_x = \partial_x u^m$ and $v^m_x = \partial_x v^m$.
Because
the equations themselves are defined for $x \in \Omega$,  the weak formulation of these problems is defined based on the $L^2(\Omega)$ inner product. Therefore, our result is only valid on $\Omega$.
\end{proof}

\begin{Remark}
Notice that in Lemma \ref{l:derivatives} we do not conclude that $u^m,v^m$ are in 
the space  $H^1(\tilde{\Omega})$, (recall $\tilde{\Omega} = \Omega \cup \Omega_0$), since this would
imply that the functions are continuous across the boundary $\partial \Omega$,
contradicting the results from \cite{du2019}.
 \end{Remark}

Lemmas \ref{l:bounds} and \ref{l:bounds_dx} stated next show that on the time interval $[0,T]$, the solutions 
to systems \eqref{e:weakfinite1} and  \eqref{e:weakfinite1N}, of the form \eqref{e:galerkin}, remain bounded in the $L^2(\Omega)$ norm with a bound that is independent of $m$.

\begin{Lemma}\label{l:bounds}
For fixed $m$, assume that $u^m, v^m$ are as in \eqref{e:galerkin}, and that they represent the first two components of the solution vector to systems
\eqref{e:weakfinite1} and \eqref{e:weakfinite1N}.
Then, there exists a small $\eps>0$ such that
\[
 \| u^m(t) \|_{L^2(\Omega)} +  \| v^m(t)\|_{L^2(\Omega) } \leq  \left(3+\sqrt{\frac{C}{A}}\right ) \left( \| u(0) \|_{L^2(\Omega)} + \| v(0) \|_{L^2(\Omega)} \right)   + \sqrt{B/A}+ 2\sqrt{| \Omega|},
 \]
 for all $t$ in $[0,T]$, with $A,B,C>0$  given by
\begin{align*}
A =& \left[ (1-\eps^2)(f+\kappa) -   \eps^2 |d_v-d_u| \right], \\
B = & \frac{1}{\eps^2} (| d_v - d_u| + \kappa ) | \Omega|   + \frac{f}{\eps^2} |\Omega|, \\
C = & \frac{1}{\eps^2} ( |d_v - d_u| + \kappa ).
\end{align*}
\end{Lemma}

\begin{proof}
To improve exposition, from now on we drop the superscript $m$ from our notation and we let
 $\| \cdot \|= \| \cdot \|_{L^2(\Omega)}$ and $( \cdot, \cdot) = (\cdot, \cdot)_{L^2(\Omega)}$. 
 Multiplying the first equation of system \eqref{e:weakfinite1} (or \eqref{e:weakfinite1N} ) by $d_k$ and adding from $k=1$ through $m$, 
 we find that
\[ \frac{1}{2} \frac{d}{dt} \| u\|^2 + d_u B_{D,N}[u,u]  = - (\sigma(u) \sigma(v)^2, u)  + f( 1-u,u).\]
Because $ f( 1-u,u) = \frac{f}{2} \left[ - \| 1- u\|^2 - \|u\|^2 + |\Omega | \right]$,
while $(\sigma(u) \sigma(v)^2, u) >0 $,
 this expression simplifies to
\[ \frac{1}{2} \frac{d}{dt} \| u\|^2 + d_u B_{D,N}[u,u]  + \frac{f}{2} \|u\|^2 \leq \frac{f}{2} | \Omega|.\]
At the same time, since the bilinear form is nonnegative, i.e. $B_{D,N}[u,u] \geq 0$, we also find that  
\[ \frac{1}{2} \left[  \frac{d}{dt} \| u\|^2 + f \|u\|^2 \right] \leq \frac{f}{2} | \Omega|.\]
Then, using Gr\"onwall's inequality, we obtain that the expression
\begin{equation}\label{e:boundu}
 \| u(t) \|^2 \leq  \| u(0)\|^2 + |\Omega|
\end{equation}
holds for all $t \in [0,T]$.

Next, we define $w = u+ v$ and derive an upper bound for its $L^2$ norm, $\|w(t)\|$. 
Adding the first two equations of system \eqref{e:weakfinite1}
 and rearranging terms gives us
the following equation for $w$,
\begin{align*}
 ( w_t, \phi_k )_{L^2(\Omega)} & + d_v B_{D,N}[w, \phi_k]  + (f+\kappa) ( w, \phi_k)_{L^2(\Omega)}  \\
  & =  (d_v-d_u) B_{D,N}[u, \phi_k] + \kappa(u, \phi_k)_{L^2(\Omega)}  + (f, \phi_k)_{L^2(\Omega)}.
\end{align*}
Then, multiplying the equation by $d_k + e_k$ and adding from $k=1$ through $m$, 
we find that
\begin{align*}
\frac{1}{2}\frac{d}{dt} \|w\|^2 & + d_v B_{D,N}[w, w]  + (f+\kappa) \|w\|^2  \\
  & =  (d_v-d_u) B_{D,N}[u, w] + \kappa(u, w)_{L^2(\Omega)}  + (f, w)_{L^2(\Omega)}.
\end{align*}
This leads us to the next set of inequalities
\begin{align*}
 \frac{1}{2} \frac{d}{dt} \|w\|^2  + (f+\kappa) \|w\|^2 
&  \leq |d_v-d_u| \|u\| \|w\| + \kappa \|u\|\|w\|  + f|\Omega|^{1/2} \| w\| \\
& \leq  ( |d_v-d_u| + \kappa) \left[ \frac{1}{\eps^2} \|u\|^2 + \eps^2 \|w\|^2 \right]
+ f \left[ \frac{1}{\eps^2} |\Omega| + \eps^2 \|w \|^2 \right],
\end{align*}
where the first line follows from the boundedness of the bilinear form $B_{D,N}$ and an
application of the Cauchy-Schwartz inequality,
 while the second line is obtained by applying  Cauchy's inequality with $\eps$.

Next, by rearranging terms and using the bound for $\|u\|^2$ found above, we arrive at
\[  \frac{1}{2}  \frac{d}{dt} \|w\|^2 + A \| w\|^2 \leq B + C\|u(0)\|^2\]
where 
\begin{align*}
A =& \left[ (1-\eps^2)(f+\kappa) -   \eps^2 |d_v-d_u| \right], \\
B = & \frac{1}{\eps^2} ( |d_v - d_u| + \kappa ) | \Omega|   + \frac{f}{\eps^2} | \Omega|,\\
C = & \frac{1}{\eps^2} ( |d_v - d_u| + \kappa ) .
\end{align*}
Notice that we can always pick $\eps$ small enough so that all these constants are positive.
Applying Gr\"onwall's inequality we then find that the following expressions
\begin{align*}
\|w(t)\|^2 &\leq  \|w(0)\|^2 + \frac{B}{A} + \frac{C}{A} \| u(0)\|^2,\\
\|w(t)\| &\leq \left(1 +  \sqrt{ \frac{C}{A}} \right)  \|w(0)\| +\sqrt{ \frac{B}{A}},
\end{align*}
hold for all $t \in [0,T]$. Because $\| w \| = \| u + v||$, using the reverse triangle inequality and the above bounds for $\|u(t)\|$, 
we finally obtain,
\[ \| v(t) \| \leq \Big ( 2 + \sqrt{ \frac{C}{A}}  \Big)( \|u(0)\| + \|v(0)\|) +\sqrt{ \frac{B}{A}} + \sqrt{| \Omega|}.\]
The results of the lemma then follow by adding the $L^2$ bounds for $u(t)$ and $v(t)$.
\end{proof}

The following lemma show that the solutions $u^m_x,v^m_x$ to the second set of equations in
\eqref{e:weakfinite1} and \eqref{e:weakfinite1N} are bounded in the $L^2$ norm.


\begin{Lemma}\label{l:bounds_dx}
For a fixed $m$, suppose that $u^m_x, v^m_x$ are as in \eqref{e:galerkin}, and that they represent the solution to the last two components of the solution vector to system \eqref{e:weakfinite1} or \eqref{e:weakfinite1N}.
Then, 
  \[ \| u_x(t)\|_{L^2(\Omega)} + \| v_x(t)\|_{L^2(\Omega)} \leq 
  \sqrt{2} \left( \| u_x(0)\|_{L^2(\Omega)} + \| v_x(0)\|_{L^2(\Omega)} +  \sqrt{ \frac{ \tilde{D} }{2E}   } \right) \rme^{ ET},\]
  for all $t \in [0,T]$, where the constant $M$ is defined as in \eqref{e:sigma}, and
  \begin{align*}
  E = & \left( \frac{\tilde{C}}{2}+ 14M^2 - f \right) \\
  \tilde{C} = &  \| \partial_x \gamma \|_{L^1(\R)} \max(du,dv) \quad \mbox{Dirichlet case} \\
  \tilde{C} = &  \| \partial_x \gamma_R \|_{L^1(\R)} \max(du,dv) \quad \mbox{Neumann case} \\
  \tilde{D}= & \frac{\tilde{C}}{2} \left( \left(3+\sqrt{\frac{C}{A}}\right ) \left( \| u(0) \|_{L^2(\Omega)} + \| v(0) \|_{L^2(\Omega)} \right)   + \sqrt{B/A}+ 2\sqrt{| \Omega|} \right).
  \end{align*}
  with $A,B,C>0$ defined in Lemma \ref{l:bounds}.
 \end{Lemma}
 
 \begin{proof}
 Because the proof is similar for both systems, we illustrate the analysis using only \eqref{e:weakfinite1}.
 As in the previous proof, to simplify our notation we drop the superscript $m$, and just write $u, v, u_x,$ and $v_x$. 
We also let $\| \cdot \|= \| \cdot \|_{L^2(\Omega)}$ and $(\cdot, \cdot ) =(\cdot, \cdot)_{L^2(\Omega)}$.
We then multiply  the third and fourth equations of system \eqref{e:weakfinite1}  by $\tilde{d}_k$ and $\tilde{e}_k$, respectively, and add these from $k=1$ through $m$. This gives us two new equations that we then add together to
find,
\begin{align*}
\frac{1}{2} \frac{d}{dt} \| u_x\|^2 + \frac{1}{2} \frac{d}{dt} \| v_x\|^2  + &  d_u \Gamma \| u_x \|^2  + d_v\Gamma \|v_x\|^2
=  d_u (J_D u, u_x) + d_v(J_Dv, v_x) \\
& -( \sigma'(u) \sigma(v)^2 u_x, u_x) -2 (\sigma(u) \sigma(v) \sigma'(v) v_x, u_x) - f \| u_x\|^2\\
& + ( \sigma'(u) \sigma(v)^2 u_x, v_x) +2 (\sigma(u) \sigma(v) \sigma'(v) v_x, v_x) - (f+\kappa) \| v_x\|^2.
\end{align*}
 
 Because the constants $\Gamma, d_u,$ and $ d_v$ are nonnegative and  since both  terms, 
$ (\sigma'(u) \sigma(v)^2 u_x ,u_x)$ and $  \kappa \|v_x\|^2$, are positive, using equation \eqref{e:sigma}, we obtain
 \begin{align*}
\frac{1}{2} \frac{d}{dt} \| u_x\|^2 + \frac{1}{2} \frac{d}{dt} \| v_x\|^2   + f( \|u_x\|^2 + \|v_x\|^2)
& \leq  d_u | (J_Du,u_x)| + d_v |(J_Dv,v_x)|\\
& + \; 2 | (\sigma(u) \sigma(v) \sigma'(v) v_x, u_x) | + | ( \sigma'(u)  \sigma(v)^2 u_x, v_x) | \\
\nonumber
&+2 |(\sigma(u) \sigma(v) \sigma'(v) v_x, v_x)|,\\[1ex]
\leq 
&\;  \| \partial_x \gamma\|_{L^1(\R)} \max(d_u,d_v) \left( \|u_x \| \|u\| + \|v_x \| \|v\| \right)\\
& \;+  8M^2 (v_x,u_x) +  4M^2  (u_x,v_x) + 8 M^2 \|v_x\|^2\\[1ex]
\leq & \; \left(\frac{\tilde{C}}{2} + 14 M^2\right) ( \|u_x\|^2 + \|v_x\|^2) + \frac{\tilde{C}}{2} (\|u\|^2 + \|v\|^2)\\
\leq & \; \left(\frac{\tilde{C}}{2} + 14 M^2\right) ( \|u_x\|^2 + \|v_x\|^2) + \tilde{D}.
\end{align*}
 Here, the second inequality comes from applying Cauchy-Schwartz to the integral terms, followed by Young's inequality for convolutions. 
 The third inequality then results from applying the AM-GM inequality to the terms $(u_x,v_x)$, $\|u\| \|u_x\|$, and  $\|v\|\|v_x\|$ and the definition of $\tilde{C}$ stated in this lemma. Finally, in the last line we pick the constant $\tilde{D}$ so that
 \[\sup_{t \in [0,T]}   \frac{\tilde{C}}{2} (\| u(t)\|^2 + \| v(t)\|^2) \leq \tilde{D} =  \frac{\tilde{C}}{2}\left[ \left(3+\sqrt{\frac{C}{A}}\right ) \left( \| u(0) \|_{L^2(\Omega)} + \| v(0) \|_{L^2(\Omega)} \right)   + \sqrt{B/A}+ 2\sqrt{| \Omega|} \right].\]
 Here, the inequality follows from Lemma \ref{l:bounds}, and the constants $A,B,C>0$ are also defined as in this Lemma.
 
 Next, letting $w(t) = \| u_x(t)\|^2 + \| v_x(t)\|^2$ and $2E= \left(\tilde{C} + 28 M^2 -2f \right)>0$, the above inequality can be written as
\begin{align*}
 \frac{d}{dt} w - 2E w \leq \tilde{D},\\
 \frac{d}{dt} \left[ w \rme^{-2Et} \right] \leq \tilde{D} \rme^{-2Et},
 \end{align*}
since $\rme^{-2Et} >0$.
It then follows from Gronwall's inequality that
\[ w(t) \leq \left (w(0) + \frac{\tilde{D}}{2E} - \frac{\tilde{D}}{2E} \rme^{-2Et} \right ) \rme^{2Et} \leq \left(w(0) + \frac{\tilde{D}}{2E}\right) \rme^{2ET}. \]
 Given that $|a| + |b| \leq \sqrt{2} \sqrt{ a^2 + b^2} \leq \sqrt{2} (|a| + |b|)$ for any $a,b\in \R$,
this then leads to the desired result,
\[ \| u_x(t)\| + \| v_x(t)\| \leq \sqrt{2} \left( \| u_x(0)\| + \| v_x(0)\| + \sqrt{\frac{\tilde{D}}{2E}}\right) \rme^{E T}.\]
\end{proof}

We are now ready to prove the main result of this section. 

\begin{Proposition}\label{p:infinitybound}
Consider \blue{systems \eqref{e:weakfinite1} and \eqref{e:weakfinite1N} } with $\sigma$ the identity map, and initial conditions $u_0,v_0,u_{x,0},v_{x,0}$ satisfying,
$u_{x,0} = \partial_x u_0, v_{x,0} =\partial_x v_0 $. Then, there exists positive constants $C_1, C_2$, 
\blue{ and a time $T$}, such that if
\[ \|u_0\|_{L^2(\Omega)} + \| v_0\|_{L^2(\Omega)} < C_1, \qquad \| u_{x,0} \|_{L^2(\Omega)} + \| v_{x,0}\|_{L^2(\Omega)} < C_2,\]
 solutions, $u^m,v^m,u^m_x,v^m_x$, \blue{to equations \eqref{e:weakfinite1} or \eqref{e:weakfinite1N} }  of the form \eqref{e:galerkin} exist for $t \in [0,T]$, are unique, and remain bounded in the $L^2(\Omega)$ norm, uniformly in $m$. Moreover, $u^m_x =\partial_x u^m, v_x = \partial_x v^m $ on $\Omega$, for all $t \in [0,T]$, and
\[ \| u^m(t) \|_\infty + \| v^m(t) \|_\infty <\infty ,\qquad \forall t\in [0,T ] , \; \forall m \in \N. \]
\end{Proposition}

\begin{proof}
Consider the system \eqref{e:weakfinite1} or \eqref{e:weakfinite1N}, with $\sigma$ as in \eqref{e:sigma} and with
 initial conditions as stated in this proposition.
Then, as shown in the introduction to this section, for any $M$ in the definition of the map $\sigma$,
there is a time $\tilde{T}$ such that this system has unique solutions, \blue{ $u^m, v^m, u^m_x, v^m_x$,}  
valid on $[0,\tilde{T}]$ and satisfying 
\blue{ $\partial_x u^m, = u^m_x,  \partial_x v^m = v^m_x$,}  (see Lemma \ref{l:derivatives}).
Lemmas \ref{l:bounds} and \ref{l:bounds_dx} then shown that these solutions are  bounded in the
$L^2(\Omega)$ norm for all $t \in [0,\tilde{T}]$, and that this bound is independent of $m$.

In addition, from Sobolev embeddings we know
that there is a constant $\bar{C}>0$, depending only on the domain $\Omega$, 
such that if $u^m(t)$ and $v^m(t)$ are in $H^1(\Omega)$ for all $t \in [0,\tilde{T}]$, then
\begin{align*}
 \| u^m(t)\|_\infty+ \| v^m(t) \|_\infty &\leq \tilde{C} ( \| u^m (t)\|_{H^1(\Omega)} + \| v^m(t) \|_{H^1(\Omega)}) \\
 & \leq  \underbrace{  \bar{C} \Big[ \| u^m(t) \|_{L^2(\Omega)} + \| v^m(t) \|_{L^2(\Omega)} \Big]}_{D_1} 
 +  \underbrace{ \bar{C} \Big[ \| u^m_x(t) \|_{L^2(\Omega)} + \| v^m_x(t) \|_{L^2(\Omega).} \Big] }_{D_2}
\end{align*}
Lemmas \ref{l:bounds} and \ref{l:bounds_dx} then give us the following bounds for $D_1$ and $D_2$,

 \begin{align*}
D_1/\bar{C} \leq &  \left(3 +\sqrt{\frac{C}{A}} \right) \left( \| u_0 \|_{L^2(\Omega)} + \| v_0 \|_{L^2(\Omega)} \right) + 2\sqrt{| \Omega|} + \sqrt{B/A} \\
D_2/\bar{C} \leq & \sqrt{2} \left( \| u_{x,0}\|_{L^2(\Omega)} + \| v_{x,0}\|_{L^2(\Omega)} +  \sqrt{ \frac{ \tilde{D}}{2E}   } \right) \rme^{ ( \tilde{C}/2 +14M^2  - f) \tilde{T}},
\end{align*}
Here the constants $A,B$ and $C$ are defined in Lemma \ref{l:bounds}, while the constants $\tilde{C}, \tilde{D}$, and $E$ are given in Lemma \ref{l:bounds_dx}. After rearranging, these bounds lead to
\begin{align*}
 \| u^m(t)\|_\infty+ \| v^m(t) \|_\infty & \leq  \bar{C} \left(3 +\sqrt{\frac{C}{A}} \right) \left( \| u_0 \|_{L^2(\Omega)} + \| v_0 \|_{L^2(\Omega)} \right)\\
 & + \bar{C} \sqrt{2} \left( \| u_{x,0}\|_{L^2(\Omega)} + \| v_{x,0}\|_{L^2(\Omega)} \right) \rme^{ ( \tilde{C}/2 +14M^2  - f) \tilde{T}} \\
 & + \bar{C} \left(   2\sqrt{| \Omega|} + \sqrt{B/A}  + \sqrt{ \frac{ \tilde{D}}{2E}} \rme^{ ( \tilde{C}/2 +14M^2  - f) \tilde{T}} \right).
 \end{align*}

We can then pick 
 $L^2(\Omega)$ bounds for the initial conditions, $u_0,v_0,u_{x,0}, v_{x,0}$, a constant $M$ and a time $T < \tilde{T}$,  such that each term in the right hand side of the above expression is less that $M/3$. Consequently,
\[
 \| u^m(t)\|_\infty+ \| v^m(t) \|_\infty < M\]
as desired.
\end{proof}

Given the results of Proposition \ref{p:infinitybound}, from now on we may consider only systems \eqref{e:weakform} and \eqref{e:weakformN}.
In addition, whenever we reference 
equations \eqref{e:weakfinite1} or equations \eqref{e:weakfinite1N}, we will always assume that the map $\sigma$ is given by the identity.

\subsection{Convergence of Approximating Sequences}\label{ss:bounds}
Our goal in this section is to show that the sequence of approximating solutions 
$\{ u^m, v^m\}$ has a strongly convergent subsequence in $L^2(0,T,L^2(\Omega))$.
To do this, we first  establish some additional energy estimates and then use Aubin's compactness theorem.

The following proposition shows that each sequence
$\{u^m\},\{v^m\},\{u_x^m\},\{v_x^m\},$ is
bounded in $L^2(0,T,V)$ uniformly in $m$. 
Moreover, it also establishes that each sequence of time derivatives, 
$\{u_t^m\},\{v_t^m\},\{(u_x)_t^m\},\{(v_x)_t^m\},$ is
is also bounded in $L^2(0,T,V')$, and in $L^2(0,T,H^{-1}(\Omega))$.

\begin{Proposition}\label{p:energybounds}
For any integer $m $, let $( u^m,v^m,u^m_x,v^m_x )$ denote the solutions to  system \eqref{e:weakfinite1}
\blue{ or \eqref{e:weakfinite1N},}
of the form \eqref{e:galerkin}.
Then, there exist a  
positive constant $D_1$,
 depending on $\| u_0\|_{L^2(\Omega)},$ $ \|v_0\|_{L^2(\Omega)},$
 $ \| u_{x,0}\|_{L^2(\Omega)}, $ $\| v_{x,0} \|_{L^2(\Omega)}$, the parameters $ |\Omega|, d_u,d_v, f,k$, 
 and a  time $T,$
   such that  
\[\begin{array}{c c c}
 \|u^m\|_{L^2( 0,T, V)} &
+ & \|u_x^m\|_{L^2( 0,T, V)} 
 \leq D_1, \\
 \|v^m\|_{L^2( 0,T, V)} &
+ &\|v_x^m\|_{L^2( 0,T, V)} 
 \leq D_1,
 \end{array}
 \qquad
 \begin{array}{c c c}
\|u^m_t\|_{L^2( 0,T, V')} &
+ & \|(u_x^m)_t\|_{L^2( 0,T, V')} 
  \leq D_1,\\
\|v^m_t\|_{L^2( 0,T, V')} &
+ & \|(v_x^m)_t\|_{L^2( 0,T, V')}  
 \leq D_1,\\
\end{array}
\]
where $V = V_D$ in the case of nonlocal Dirichlet boundary conditions, and $V= V_N$ in the case of nonlocal Neumann boundary conditions. 

In addition, we also obtain the existence of a constant $D_2>0$,
depending on the initial conditions, the system parameters, and a time $T$, such that
\[
\begin{array}{ c c}
\|u^m_t\|_{L^2( 0,T, H^{-1}(\Omega))}  & \leq D_2,\\
\|v^m_t\|_{L^2( 0,T, H^{-1}(\Omega))}  & \leq D_2.\\
\end{array}
\]
\end{Proposition}

\begin{proof}

Recall first that the spaces $V_D$ and $V_N$ are equivalent to $L^2(\Omega)$.
The $L^2(0,T,V)$ bounds for $u^m,v^m,u_x^m,$ and $ v^m_x$ then
follow from Lemmas \ref{l:bounds} and \ref{l:bounds_dx}. We therefore only need to show 
that the time derivatives, $u^m_t, v^m_t , (u^m_x)_t,$ and $ (v^m_x)_t$ 
are bounded in $L^2(0,T, V')$, and that $u^m_t$ and $ v^m_t$ are bounded in $L^2(0,T, H^{-1}(\Omega))$.
We start with the $L^2(0,T, V')$ bounds.

As in  previous proofs, we drop the superscript $m$ from our notation and use $\| \cdot \| = \| \cdot \|_{L^2(\Omega)}$.
We let $\psi$ be an element in the space $V$ with $\| \psi\|=1$.
Using the finite set $\{ \phi_k\}_{k=0}^{2m}$ we decompose this function into two parts, 
$\psi = \psi_1 + \psi_2$, with $\psi_1 \in \spann\{ \phi_k\}$, and $\psi_2 \in (\spann\{ \phi_k\})^\perp$. 
It is then clear that $\| \psi_1 \|\leq \| \psi\| = 1$.

 Starting with the Dirichlet case, and looking at the first equation in \eqref{e:weakfinite1}, we obtain
\[ \langle u_t, \psi \rangle_{V'\times V} = ( u_t, \psi_1)_{L^2(\Omega)} = (F^{(1)}(u,v), \psi_1)_{L^2(\Omega)} - d_u B_{D}[ u,\psi_1], \]
where the first equality follows from $V \subset L^2(\Omega) \subset V'$ being a Gelfand triple, and the definition of $u = u^m$. Notice that the $L^2(0,T, V')$ bounds for $u_t$
 follow from finding a constant $\alpha>0$ such that
\[ (F^{(1)}(u,v), \psi_1)_{L^2(\Omega)} \leq \alpha \|\psi_1\|,\]
holds for all $t \in [0,T]$. Indeed, if this condition is satisfied, then
\[
| \langle u_t, \psi \rangle_{V'\times V} |   \leq \alpha \| \psi_1\| + d_u \| u \| \| \psi_1\| \leq   \alpha + d_u \| u \|. 
\]
As a result,
\[\| u_t\|_{V'}   \leq  \alpha + d_u \| u \|, \]
and it then follows that,
\[
\int_0^T \| u_t \|^2_{V'} \;dt  \leq \int_0^T   (\alpha^2  + d^2_u \| u \|^2 ) 
 \leq  T  (\alpha^2  + d^2_u ( \| u(0)\|^2 + | \Omega|) \leq D,\]
where the  second line follows from expression \eqref{e:boundu} in the proof of Lemma \ref{l:bounds} and $D$ represents a positive constant. 

Similarly, with the same notation as above we can write the third equation in \eqref{e:weakfinite1} as
\[ \langle (u_x)_t, \psi \rangle_{V'\times V} = ( u_t, \psi_1)_{L^2(\Omega)} = (F^{(1)}(u,v), \psi_1)_{L^2(\Omega)} - d_u \Gamma (u_x, \psi_1) + d_u (J_D u, \psi_1),
\]
from which it follows that
\[ | \langle (u_x)_t , \psi \rangle | \leq \alpha \|\psi_1\| + d_u \Gamma \|u_x\| \|\psi_1 \|+ 
d_u \| \partial_x \gamma \|_{L^1(\R)} \|u\| \|\psi_1\|. \]
Consequently,
\[ \|(u_x)_t\|_{V'} \leq \alpha + d_u \Gamma \|u_x\| +d_u \| \partial_x \gamma \|_{L^1(\R)} \|u\| . \]
and 
\begin{align*}
\int_0^T \| u_t \|^2_{V'} \;dt & \leq 4 \int_0^T   (\alpha^2  + d^2_u \Gamma^2 \| u_x \|^2 
+ d_u^2\| \partial_x \gamma \|^2_{L^1(\R)} \|u\|^2  ) \\
\int_0^T \| u_t \|^2_{V'} \;dt & \leq 4 T  \Big (\alpha^2  + d^2_u \Gamma^2 \bar{D}+ d_u^2\| \partial_x \gamma_R \|^2_{L^1(\R)}  \left ( \| u(0)\|^2 + | \Omega| \right)\Big)  \\
\int_0^T \| u_t \|^2_{V'} \;dt & \leq D,
\end{align*}
where as before the constant $D>0$ represents a bound for the right hand side appearing in the second line, the bounds on $\|u\|$ follow from expression \eqref{e:boundu} in the proof of Lemma \ref{l:bounds},
and the constant $\bar{D}$ corresponds to the bound for $\|u_x\|^2$ and is based on the results stated in Lemma \ref{l:bounds_dx}.

Notice that a similar analysis gives the bounds for the remaining quantities $v_t, (v_x)_t$, as well as for the solutions to \eqref{e:weakfinite1N},
 provided we
 find bounds for all nonlinear terms $F^{(i)}, G^{(i)}$, $i =1,2$.
Because the details are very similar in all cases, we only show the details for the nonlinear terms $F^{(1)}$ and $G^{(1)}$.

First, thanks to H\"older's inequality,
\begin{align*}
| (F^{(1)}(u,v), \psi_1 )_{L^2(\Omega)} | & = ( -uv^2 + f(1-u) , \psi_1)_{L^2(\Omega)} \\
& \leq \|u \|_\infty \|v\|_\infty \|v \| \| \psi_1\| + f  |\Omega|^{\frac{1}{2}} \| \psi_1 \| + f \| u\| \| \psi_1\|\\
& \leq ( \|u \|_\infty \|v\|_\infty \|v \| + f |\Omega| ^{\frac{1}{2}}+ f \| u\| ).
\end{align*}

Using Proposition \ref{p:infinitybound} and Lemma \ref{l:bounds}, we then find that the nonlinearity $F^{(1)}$ is indeed
bounded by a constant that depends on the initial conditions and the parameters of the problem.
Similarly,
\begin{align*}
| (G^{(1)}(u,v,u_x,v_x), \psi_1 )_{L^2(\Omega)}|  & = ( -u_xv^2  -2 uv v_x - f u_x , \psi_1)_{L^2(\Omega)} \\
& \leq  \| v\|_\infty^2 \| u_x \| \| \psi_1 \| + 2 \| u\|_\infty \| v\|_\infty \|v_x \| \| \psi_1 \| + f\| u_x \| \| \psi_1\|\\
& \leq  ( \| v\|_\infty^2 \| u_x \| + 2 \| u\|_\infty \| v\|_\infty \|v_x \|  + f\| u_x \| ).
\end{align*}
and the result again follows from Proposition \ref{p:infinitybound} and Lemma \ref{l:bounds_dx}.

Finally, we show that the time derivatives $u^m_t , v^m_t$ \blue{of the solutions $u^m, v^m$ to equations \eqref{e:weakform} and \eqref{e:weakformN},}  are also bounded in $L^2(0,T, H^{-1}(\Omega))$.
This result follows from {\color{blue} Lemma \ref{l:derivatives}  }
which shows that the solution vector satisfies 
\[u^m_x = \partial_x u^m, \quad v^m_x= \partial_x v^m,\]
on $\Omega$, together with the fact that the set $\{\phi_k\}$ is also a basis for $H^1(\Omega)$ and that the space  $H^1(\Omega) \subset L^2(\Omega) \subset H^{-1}(\Omega)$. As a result, the above proofs also hold in the case when $V' = H^{-1}(\Omega)$.
\end{proof}

Using the results from the above proposition, together with Alaogu's theorem, we may now extract  weakly convergent subsequences, which for simplicity we do not relabel:
\begin{equation}\label{e:c1}
\begin{array}{l r}
u^m \rightharpoonup \bar{u} &  \\[-1ex]
& \quad \mbox{in} \qquad L^2(0,T, V),\\
v^m \rightharpoonup \bar{v} &  
\end{array}
\hspace{10ex}
\begin{array}{l r}
u^m_t \rightharpoonup \bar{u}_t &  \\[-1ex]
& \quad \mbox{in} \qquad L^2(0,T, V'),\\
v^m_t \rightharpoonup \bar{v}_t &  
\end{array}
\end{equation}

 \begin{equation}\label{e:c2}
\begin{array}{l r}
u_x^m \rightharpoonup \bar{u}_x&  \\[-1ex]
& \quad \mbox{in} \qquad L^2(0,T, V),\\
v_x^m \rightharpoonup \bar{v}_x &  
\end{array}
\hspace{8ex}
\begin{array}{l r}
(u_x^m)_t \rightharpoonup (\bar{u}_x)_t &  \\[-1ex]
& \quad \mbox{in} \qquad L^2(0,T, V').\\
(v^m_x)_t \rightharpoonup (\bar{v}_x)_t &  
\end{array}
\end{equation}

Notice that because the spaces $V = V_D,V_N$ are equivalent to $L^2(\Omega)$, then Proposition \ref{p:infinitybound} implies
 that when restricted to the domain $\Omega$ the sequences $\{ u^m\}, \{v^m\},$ are also bounded in 
 $L^2(0,T,H^1(\Omega))$.  
Therefore, we may extract weakly convergent subsequences  in $L^2(0,T, H^1(\Omega))$,
and in the case of the time derivatives in $L^2(0,T, H^{-1}(\Omega))$. That is,
\begin{equation}\label{e:c3}
\begin{array}{l r}
u^m \rightharpoonup \bar{u} &  \\[-1ex]
& \quad \mbox{in} \qquad L^2(0,T, H^1(\Omega)),\\
v^m \rightharpoonup \bar{v} &  
\end{array}
\hspace{10ex}
\begin{array}{l r}
u^m_t \rightharpoonup \bar{u}_t &  \\[-1ex]
& \quad \mbox{in} \qquad L^2(0,T, H^{-1}(\Omega)).\\
v^m_t \rightharpoonup \bar{v}_t &  
\end{array}
\end{equation}
This last point, combined with Aubin's compactness theorem (see \cite{aubin1963} and the Appendix), then implies that 
there is a further subsequence which converges strongly in $L^2(0, T, L^2(\Omega))$.

\subsection{Main Theorem}\label{ss:main}
We are now ready to show the existence of solutions to the weak formulation of our problem.
More precisely we prove the following theorem.
\begin{Theorem**}  
Let $u_0,v_0,\partial_x u_0,\partial_x v_0$ be in $L^2(\Omega)$. Then,
there exists positive constants $C_1,C_2$ and  $T$, such that if
\[ \| u_0\|_{L^2(\Omega)} + \| v_0\|_{L^2(\Omega)} <C_1,\quad \| \partial_x u_0\|_{L^2(\Omega)} + \| \partial_x v_0\|_{L^2(\Omega)} <C_2,\]
 the system of equations
\begin{align} 
\label{e:uweak}
\langle u_t, \phi \rangle_{(V' \times V)} + d_u B_{D,N} [ u, \phi] & =  ( - uv^2 + f(1-u), \phi)_{L^2(\Omega)},\\
\label{e:vweak}
\langle v_t, \phi \rangle_{(V' \times V)} + d_v B_{D,N} [ v, \phi] & =  (  uv^2 - (f+\kappa )v), \phi)_{L^2(\Omega)},
\end{align}
has a unique weak solution  $(u,v) \in [L^2(0,T, H^1(\Omega)] \times [L^2(0,T, H^1(\Omega)] $ 
valid on the time interval $[0,T]$,
and satisfying $u(x,0) = u_0$ and $v(x,0)=v_0$.

\end{Theorem**}
\begin{proof}
Let $u^m,v^m \in V$ be the first two components of the solution to systems \blue{ \eqref{e:weakfinite1} and \eqref{e:weakfinite1N}} of the form \eqref{e:galerkin}. 
In the Dirichlet case we assume $V = V_D$, while in the Neumann case we let $V = V_N$, (see Definitions \ref{e:D} and \ref{e:N}, respectively).
From Propositions \ref{p:infinitybound} and \ref{p:energybounds}, 
 we know that there is a constant $T$ such that these functions are solutions to the  equations \eqref{e:uweak} and \eqref{e:vweak}
 that are bounded in 
$L^2(0, T,H^1(\Omega)) \subset L^2(0,T,V)$, 
 and with time derivatives that are bounded in 
 $L^2(0, T,V')$. 
 Thus, we may extract weakly convergent subsequences, as in \eqref{e:c1}.
 
 Fix an integer $N$ and choose a function $w \in C^1([0,T], V)$ of the form
 \begin{equation}\label{e:w}
 w = \sum_{k=0}^N a_k(t) \phi_k(x),
 \end{equation}
 with $a_k(t)$ smooth. Let $\{\phi_k\}_{k=0}^{2m}$ be a set of basis elements with $2m\geq N$.
Multiplying equations 
 \eqref{e:uweak} and \eqref{e:vweak} by $a_k(t)$, summing them from $k=1$ through $N$, and integrating with respect to time
 gives us the following system,
 which is valid for all $m$,
 \begin{align}\label{e:weaktimeintegral}
 \int_0^{T} \langle u^m_t, w \rangle + d_u B_{D,N}[u^m, w] \;dt & 
 =  \int_0^{T} (F^{(1)}(u^m,v^m,w)_{L^2(\Omega)} \;dt,\\
 \nonumber
 \int_0^{T} \langle v^m_t, w \rangle + d_v B_{D,N}[v^m, w] \;dt & 
 =  \int_0^{T} (F^{(2)}(u^m,v^m,w)_{L^2(\Omega)} \;dt.
\end{align}
It is straight forward to check that  by passing to a subsequence the left hand sides converge to
 \begin{align*}
 \int_0^{T} \langle u^m_t, w \rangle + d_u B_{D,N}[u^m, w] \;dt & \longrightarrow \int_0^{T} \langle \bar{u}_t, w \rangle  + d_u B_{D,N}[\bar{u}, w] \;dt, \\
 \int_0^{T} \langle v^m_t, w \rangle + d_v B_{D,N}[v^m, w] \;dt &\longrightarrow \int_0^{T} \langle \bar{v}_t, w \rangle  + d_v B_{D,N}[\bar{v}, w] \;dt .
 \end{align*}
Indeed, the convergence of the first terms in the two lines above follows directly from the weak convergence in $L^2(0,T, V')$ 
of the time derivatives of the approximating solutions, while a short computation, shown next,  proves the result for the second terms.

In the case of nonlocal Dirichlet boundary conditions,
\begin{align*}
B_D[u^m,w]  & = \frac{1}{2} \int_\R \int_\R (u^m(y,t) - u^m(x,t)) \gamma(x,y) (w(y,t) - w(x,t) ) \;dy\;dx\\
& = -   \int_\Omega u^m(x,t) \int_\R  (w(y,t) - w(x,t) ) \gamma(x,y)  \;dy\; dx\\
& = -  \int_{\Omega } u^m(x,t) g(x,t) \;dx ,
\end{align*}
where in the last line we defined $g(x,t) = w \ast \gamma -  \Gamma w $, with $\Gamma = \int_\R \gamma(z)\;dz>0$. We 
 check below that  $ g \in L^2(0,T, V)$, so that passing to a subsequence and taking the limit
we arrive at
$\int_0^{T} B_D[u^m,w] \rightarrow \int_0^{T}B_D[\bar{u},w]$.
Indeed, notice that
\begin{align*}
\|g\|_{L^2(0,T,L^2(\Omega))} &= \|w \ast \gamma - \Gamma w \|_{L^2(0,T,L^2(\Omega))}\\
& \leq \|w \ast \gamma \|_{L^2(0,T,L^2(\Omega))}  + \Gamma \|w \|_{L^2(0,T,L^2(\Omega))}.
\end{align*}
Since $\gamma$ is a radial function in $L^1(\R)$,
 then Young's inequality implies $$\| \gamma \ast w\|_{L^2(\Omega)} \leq \| \gamma \ast w\|_{L^2(\R)} \leq \| \gamma\|_{L^1(\R)} \|\; \| w\|_{L^2(\R)} =   \| \gamma\|_{L^1(\R)} \|\; \| w\|_{L^2(\Omega)}.$$ Since $V \simeq L^2(\Omega)$ and $w \in  L^2(0,T, V)$, it follows that $g$ is in $ L^2(0,T, V)$, as well.
 
 In the case of nonlocal Neumann boundary conditions, similar arguments show that
 $\int_0^{T}B_N[u^m, w] \rightarrow \int_0^{T} B_N[\bar{u}, w]$,
 provided that both, $u^m $ and the basis elements, $\phi_k$, in definition of $w$, are in the 
 space $V_N$ (see Lemma \ref{l:exteriorproblem}).

Next, we show that the nonlinear terms, as expressed in the right hand sides of the expression 
\eqref{e:weaktimeintegral}, also converge to the correct limit as $m$ goes to infinity.
To show that
\[ \int_0^{T}\Big (F^{(i)}(u^m,v^m),w \Big )_{L^2(\Omega)} \;dt \longrightarrow \int_0^{T}  \Big (F^{(i)}(\bar{u},\bar{v}) ,w \Big )_{L^2(\Omega)} \;dt,\]
for $i =1,2$, with
\begin{equation}\label{e:nonlinearities}
 F^{(1)}(u,v) = -uv^2 +(f-1)u \qquad F^{(2)}(u,v) = uv^2 -(f+\kappa) v,
 \end{equation}
  first notice that the only nonlinearity in the expressions for $F^{(i)}$ comes from the term $uv^2$.
  The linear terms can easily be seen to satisfy the above limit thanks to the weak
convergence of the approximating sequences $u^m,v^m$.
To show the result for the terms $uv^2$, we
 use the Dominated convergence theorem. 
 
 Since $w$ is of the form \eqref{e:w}, and we use basis functions, $\phi_k$, which are continuous on our bounded 
 domain $\Omega$, then
 \begin{align*}
 \int_0^{T} \int_\Omega | u^m(v^m)^2 w- \bar{u} \bar{v}^2w| \;dx \;dt 
 & \leq  \|w\|_{L^\infty([0,T] \times \Omega)} \int_0^{T} \int_\Omega | u^m(v^m)^2 - \bar{u} \bar{v}^2| \;dx \;dt. 
\end{align*}
From the discussion in Section \ref{ss:bounds}, the weak limits \eqref{e:c3}, and Aubin's Theorem, we know that we can extract subsequences, $\{u^{m_j}\}$ and $\{v^{m_j}\}$, which converge strongly in $L^2(0,T, L^2(\Omega))$, and sub-subsequences
(which we do not relabel) that converge also point-wise to  $\bar{u}$, $\bar{v}$.  
Hence, $u^{m_j}(v^{m_j})^2$ also converges point-wise to
 $\bar{u} \bar{v}^2 $, for almost all $(x,t) \in \Omega \times [0,T]$ .
At the same time, Proposition \ref{p:infinitybound}
 shows that $\| u^{m_j}(t) \|_{L^\infty(\Omega)}, \| v^{m_j}(t) \|_{L^\infty(\Omega)} < D$ for a.e $t \in [0,T]$ and some positive constant $D$.
As a result, $| u^{m_j}(v^{m_j})^2(x,t)|< D^3$ for a.e. $(t,x) \in [0,T] \times \Omega$, uniformly in $m$.
Since the conditions of the Dominated convergence theorem are satisfied, the convergence written in \eqref{e:nonlinearities} then follows.

We therefore find that
 \begin{align*}
 \int_0^{T} \langle \bar{u}_t, w \rangle + d_u B_{D,N}[\bar{u}, w] \;dt & =  \int_0^{T} (F^{(1)}(\bar{u},\bar{v}),w)_{L^2(\Omega)} \;dt,\\
 \int_0^{T} \langle \bar{v}_t, w \rangle + d_v B_{D,N}[\bar{v}, w] \;dt & =  \int_0^{T} (F^{(2)}(\bar{u},\bar{v}),w)_{L^2(\Omega)} \;dt.
\end{align*}
for all $w \in L^2(0,T,V)$, since the subspace $C^1(0, T,V)$ is a dense subset. Consequently
 \begin{align*}
 \langle \bar{u}_t, w \rangle + d_u B_{D,N}[\bar{u}, w]  & =   (F^{(1)}(\bar{u},\bar{v}),w)_{L^2(\Omega)}, \\
\langle \bar{v}_t, w \rangle + d_v B_{D,N}[\bar{v}, w]  & =   (F^{(2)}(\bar{u},\bar{v}),w)_{L^2(\Omega)} .
\end{align*}
for all $w \in V$ and a.e. $t \in [0,T]$.

Finally, we show that the limits $\bar{u} , \bar{v}$ satisfy the initial conditions of the problem.
Notice that from  the weak convergence \eqref{e:c3}, it follows that $\bar{u},\bar{v} \in  C^1([0,T],L^2(\Omega)) \sim C^1([0,T],V) $.
Then, assuming  $w \in C^1([0,T], V)$ with $w(T) = 0$, we see that the expressions shown above can be written as,
 \begin{align*}
 \int_0^{T} \langle -\bar{u}, w_t \rangle + d_u B_{D,N}[\bar{u}, w]  \;dt& =  \int_0^{T}  (F^{(1)}(\bar{u},\bar{v}),w)_{L^2(\Omega)} \;dt + (\bar{u}(0), w(0))_{L^2(\Omega)},\\
 \int_0^{T} \langle -\bar{v}, w_t \rangle + d_v B_{D,N}[\bar{v}, w] \;dt  & =   \int_0^{T} (F^{(2)}(\bar{u},\bar{v}),w)_{L^2(\Omega)} \;dt + (\bar{v}(0), w(0))_{L^2(\Omega)}.
\end{align*}
Similarly, the sequence of approximations take the form,
 \begin{align*}
 \int_0^{T} - \langle u^m, w_t \rangle + d_u B_{D,N}[u^m, w] \;dt & 
 =  \int_0^{T} (F^{(1)}(u^m,v^m),w)_{L^2(\Omega)} \;dt + (u^m(0), w(0))_{L^2(\Omega)},\\
 \int_0^{T} - \langle v^m, w_t \rangle + d_v B_{D,N}[v^m, w] \;dt & 
 =  \int_0^{T} (F^{(2)}(u^m,v^m),w)_{L^2(\Omega)} \;dt +(v^m(0), w(0))_{L^2(\Omega)}.
\end{align*}
If we denote  the initial conditions for $u$ and $v$ by $g_1, g_2$, then by assumption $u^m(0) \rightharpoonup g_1$
and $v^m(0) \rightharpoonup g_2$ in $V$, or equivalently in $L^2(\Omega)$. 
 Since $w(0)$ is arbitrary, it follows that $\bar{u}(0) = g_1$ and $\bar{v}(0) = g_2$.

\end{proof}

\section{Numerical Illustrations}\label{sec:num}

In this section, we introduce and investigate the convergence properties of two numerical schemes that approximate the Gray-Scott model with Dirichlet and Neumann nonlocal boundary constraints (BC). The time discretization is based on a backward difference formula of order one, and the space discretization relies on the finite element method using Lagrange polynomial functions. We first recall the nonlocal Gray-Scott model considered on a one dimensional domain $\Omega$ over a time interval $[0,T]$:
\begin{align} \label{e:GSnum}
 u_t - d_u \blue{K u} + uv^2 -f (1-u) = & q_u, \\
 \nonumber
 v_t - d_v \blue{K v} - uv^2 + (f+\kappa )v  = & q_v.
\end{align}
We note that, unlike the equations \eqref{e:GS} studied in the rest of the paper, the above equations have been supplemented by two source terms,
$q_u$ and $q_v$, in order to investigate the convergence properties of our schemes using manufactured solutions. 
We remind the reader that in the case of nonlocal Neumann boundary constraints,
the computational domain is extended to $\tilde{\Omega}=\Omega \cup \Omega_0$, which leads us to solve the following problem:
$$ \blue{K u}= q_u, \quad  \blue{K v} = q_v \quad \text{ on } \Omega_0.$$

\subsection{Time marching} 
Let $\tau>0$ be a time step such that $\tau N=T$ with $N\geq 1$ an integer. 
We set $t_n=n\tau$ for any integer $n\geq 0$. For any time dependent function $\phi$, we set $\phi^n=\phi(t_n)$.  We introduce two time discretizations of the Gray-Scott model for nonlocal Dirichlet and Neumann boundary constraints, where first order backward formulas are used to approximate the time derivatives $\partial_t u$ and $\partial_t v$.
The nonlinear terms proportional to $uv^2$ are made explicit using first order time extrapolation. The schemes then read as follows.

\noindent First order scheme for Dirichlet problem on $\tilde{\Omega}=\Omega$ (DBDF1):
\begin{align} \label{e:GS_BDF1}
\frac{u^{n+1}-u^n}{\tau} - d_u \blue{K u^{n+1}} - f(1-u^{n+1}) = - u^n (v^n)^2 + q_u^{n+1}, \\
\nonumber
\frac{v^{n+1}-v^n}{\tau} - d_v \blue{K v^{n+1}}  + (f+\kappa )v^{n+1} = u^n (v^n)^2 + q_v^{n+1}.
\end{align}
\noindent First order scheme for Neumann problem on $\tilde{\Omega}=\Omega \cup \Omega_0$ (NBDF1):
\begin{align} \label{e:GS_BDF}
\frac{u^{n+1}-u^n}{\tau} \mathbbm{1}_\Omega  - d_u \blue{K u^{n+1}} 
- f(1-u^{n+1})\mathbbm{1}_\Omega 
= - u^n (v^n)^2\mathbbm{1}_\Omega + q_u^{n+1}, \\
\nonumber
\frac{v^{n+1}-v^n}{\tau} \mathbbm{1}_\Omega- d_v \blue{K v^{n+1}} 
+ (f+\kappa)v^{n+1}\mathbbm{1}_\Omega 
= u^n (v^n)^2\mathbbm{1}_\Omega + q_v^{n+1},
\end{align}
where the function $\mathbbm{1}_\Omega$ is equal to one in the domain $\Omega$ and zero elsewhere.

\subsection{Space discretization} 
The domain $\tilde{\Omega}$ is discretized using a sequence of conforming shape regular meshes $(E_h)_{h>0}$. Each element $E$ of the mesh $E_h$ is  an interval included in $\tilde{\Omega}$ of size $h_E$. The mesh size, denoted by $h$, is then define as $h:=\max_{E\in E_h} h_E$.
The unknowns $u,v$ are approximated using continuous piece-wise linear Lagrange nodal finite element  $\mathbb{P}_1$. Meaning that $(u,v)$ are approximated in the following finite element space:
\begin{equation} \label{eq:FEM_space}
X_h:= \{ g \in C^0(\tilde{\Omega}) \; | \; g_{|E} \in \mathbb{P}_1(E), \; \forall E \in E_h   \}.
\end{equation}

The fully discretized algorithm reads: Find $(u,v)\in X_h$ such that the following holds for all $g \in X_h$:
\begin{align} \label{e:GS_BDF1_FEM}
\int_{\Omega} \frac{u^{n+1}-u^n}{\tau}  g \;dx
- d_u \int_{\tilde{\Omega}} \blue{K u^{n+1}}  g \;dx
- \int_{\Omega} f(1-u^{n+1})  g \;dx
=
-\int_{\Omega}  u^n (v^n)^2 g \;dx
+\int_{\tilde{\Omega}}  q_u^{n+1} g\;dx , \\
\nonumber
\int_{\Omega}  \frac{v^{n+1}-v^n}{\tau}  g \;dx
- d_v \int_{\tilde{\Omega}}  \blue{K v^{n+1}}   g \;dx
+ \int_{\Omega} (f+k)v^{n+1}  g \;dx
= 
\int_{\Omega} u^n (v^n)^2  g \;dx
+\int_{\tilde{\Omega}}  q_v^{n+1} g \;dx,
\end{align}
where we recall that $\Omega=\tilde{\Omega}$ in the case of nonlocal Dirichlet boundary constraints. 
As the $L^2$ error is bounded by $\mathcal{O}(\tau+h^2)$, 
under classic CFL condition (i.e. $\tau \sim h$),  the above 
algorithm is expected to converge with order one. We refer to \cite{brenner2008mathematical,ern2004theory} 
for more information on finite element methods. 
The following sections investigate the convergence 
properties of the above algorithm when using nonlocal Dirichlet 
and Neumann boundary constraints.

\subsection{Manufactured tests with nonlocal Dirichlet boundary conditions}
We study the convergence properties in $L^2$-norm of the above algorithm with nonlocal Dirichlet boundary constraints.
The computational domain is set to $\Omega=[0,1]$ and the final time is set to $T=1$. The manufactured solutions are defined by:
$$u(x,t) =x^2 \cos(\pi x/2) e^{-x + x^2 -t}, \qquad  v(x,t) = \sin(x) (1-x) e^{-x+x^2} (10+xt)\cos(t^2)/300.$$ 
The kernel function is defined as follows
$$\gamma(x,y) = e^{-(x-y)^2}.$$
The problems parameters are respectively set to $d_u=0.05, d_v=0.01, k=2$ and $f=6$. 
The source terms $q_u$ and $q_v$ are computed accordingly. We perform a series of tests on five uniform grids 
with respective mesh size $h$ set to $0.05, 0.025, 0.0125, 0.00625$, and $0.003125$. The time step is set to $\tau=2h$.
The relative errors in the $L^2$-norm obtained  with the DBDF1 algorithm are displayed in 
Table~\ref{tab:num_test_1_BDF1}. We recover a rate of convergence equal to one, which is compatible with the
expected theoretical rate $\mathcal{O}(\tau+h^2)$. 
\begin{table}[ht]
\centering
\begin{tabular}{|c|c|c|c|c|c|c|} \hline  
 \multicolumn{3}{|c|}{$L^2$-norm of error}   & \multicolumn{2}{|c|}{$u$}
 & \multicolumn{2}{|c|}{$v$} \\ \hline \hline
&  $h$  & $n_{df}$ & Error  & Rate  & Error  & Rate \\ \hline
 \multirow{5}{*}{$\tau= 2 h$}  
 & $ 0.05$& 20 			& 1.43E-2 	& - 		& 3.72E-2		&  -    \\ \cline{2-7}
& $ 0.025$& 40 		&  6.00E-3	& 1.25	& 	1.96E-2 	&  0.92  \\ \cline{2-7}
&$0.0125$ & 80 		& 2.70E-3		& 	1.15	& 	1.01E-2 	& 0.96 \\  \cline{2-7}
&$0.00625$ & 160 	& 1.30E-3		&  1.05	& 5.10E-3		&  0.99 \\  \cline{2-7}
&$0.003125$ & 320 	& 6.29E-4		&  1.05	& 2.60E-3		& 0.97  \\ 
 \hline
\end{tabular}
\caption{Convergence tests for BDF1 algorithm with Dirichlet BC. The mesh size is denoted by $h$, the time step by $\tau$, and the degrees
of freedom per unknown by $n_{df}$.}
\label{tab:num_test_1_BDF1}
\end{table}

\subsection{Manufactured tests with nonlocal Neumann boundary conditions}
We now study the convergence properties of the algorithm for a problem with nonlocal Neumann boundary constraints, where the kernel function $\gamma$ has a finite horizon $R=2$.
The problem domain, solutions, kernel and parameters are defined as follows:
\begin{itemize}
\item $\Omega=[-8,8]$ and  $T=1$.
\item $\Omega_0=[-10,-8]\cup[8,10]$.
\item Kernel: $\gamma(x,y) = 0.5 e^{|x-y|} \mathbbm{1}_{|x-y|\leq 2}.$
\item Solution $u(x,t) = (x-10)(x+10) \cos(t)/100. $
\item Solution $v(x,t) =  \sin(\pi x/10) e^{-t^2}/2.$
\item Parameters: $d_u=0.05, \quad d_v=0.01, \quad k=3, \quad f=2$.
\end{itemize}
We perform a series of tests using the NBDF1 algorithm and five uniform grids with a mesh size $h$ varying from $0.5$ to $0.03125$, and
set the time step to $\tau=h/5$. As shown in table \ref{tab:num_test_2_BDF1}, the results
are consistent with a theoretical convergence order of  $\mathcal{O}(\tau+h^2)$. We note that the number of degrees of freedom
reported in the table \ref{tab:num_test_2_BDF1} represents the number of degrees of 
freedom on the computational domain $\Omega\cup\Omega_0$.  
\begin{table}[ht]
\centering
\begin{tabular}{|c|c|c|c|c|c|c|} \hline  
 \multicolumn{3}{|c|}{$L^2$-norm of error}   & \multicolumn{2}{|c|}{$u$}
 & \multicolumn{2}{|c|}{$v$} \\ \hline \hline
&  $h$  & $n_{df}$ & Error  & Rate  & Error  & Rate \\ \hline
 \multirow{5}{*}{$\tau= h/5$}  
& $ 0.5$& 40 		& 2.40E-3	& - 	& 4.30E-3	&  -    \\ \cline{2-7}
& $ 0.25$& 80 		& 1.40E-3	& 0.78	& 2.30E-3	& 0.90   \\ \cline{2-7}
&$0.125$ & 160 	& 7.25E-4	& 0.95	& 1.20E-3	& 0.94\\  \cline{2-7}
&$0.0625$ & 320 	& 3.72E-4	&  0.96	& 6.09E-4	&  0.98 \\  \cline{2-7}
&$0.03125$ & 640 	& 1.89E-4	&  0.98	& 3.08E-4	& 0.98 \\ 
 \hline
\end{tabular}
\caption{Convergence tests for BDF1 algorithm with Neumann BC. The mesh size is denoted by $h$, the time step by $\tau$ and the degrees
of freedom per unknown by $n_{df}$.}
\label{tab:num_test_2_BDF1}
\end{table}

\subsection{Toward pattern formation}
As mentioned in the introduction, the local Gray-Scott model is known for generating a wealth of interesting
spatio-temporal structures, \cite{pearson1993}. In the one-dimensional case, one can prove existence 
of periodic patterns as well as pulse and multi-pulse solutions (see \cite{doelman1997, morgan2000} for the respective proofs).
We use this fact as motivation for investigating
the emergence of patterns in the nonlocal Gray-Scott model.  The numerical schemes 
constructed here lend themselves well for investigating pulse and multi-pulse solutions. For
simplicity, we only focus on pulse solutions and we consider only homogeneous Neumann boundary constraints.
 Our goal is to numerically study 
the effects of nonlocal diffusion on the formation of these patterns.

To run the simulations we use as
 computational domain  $\Omega=[-40,40]$ and outer domain $\Omega_0 = [-45,-40] \cup[ 40, 45]$.
 For the kernel we consider the exponential function,
 $$\gamma(x,y) = A e^{-a|x-y|} \mathbbm{1}_{|x-y|\leq R}$$
  where  $R$ is the kernel's horizon, $a$ is a dispersive range and the parameter $A$ is picked so that the kernel's average is equal to one, i.e. $A=\dfrac{a}{2(1-e^{-aR})}.$
Physical parameters are chosen  within the range of values for which pulse solutions are known 
to exist in the 'local' Gray-Scott model. More precisely, we take
$d_u=1, \quad d_v=0.01, \quad k=0.0977,$ and $ f=0.01$.
We also set the horizon to $R=5$, and take as initial conditions 
$$u(x,0) = 1-0.3 e^{-10x^2}, \quad v(x,0) =  e^{-10 x^2}.$$

 In order to compare our results to those of the local model, 
we rescale the \blue{map $Ku$} by a factor $C$, such that the nonlocal diffusion operator $CKu$ tends to the Laplacian as the parameter $a$ tends to infinity. This leads us to define $C$ as follows:
$$C=\frac{a^3}{A(2-exp^{-aR}(1+aR(2+aR)))}.$$

To find pulse solutions, which are steady-state solutions, we set the mesh size $h = 0.05$, the time step $\tau =0.01$, and we run the scheme until the relative error between two successive iterations is les than a tolerance of $tol = 10^{-5}$.
Figure \ref{fig:pulse_neumann} shows the profile of the solution, $(u,v)$, for various values of the parameter $a$.
We find that  pulse solutions become narrower as the value of the parameter $a$ decreases. Physically, this corresponds to an increase in the dispersive range. Moreover, when $a$ is too small the peak of the pulse caves in, giving us a `batman-like' profile.
On the other hand, as the value of the parameter $a$ increases, the pulse profile resembles more and more that of the `local' Gray-Scott model, as expected.

 \begin{figure}[t] 
   \centering
  \subfigure[]{ \includegraphics[width=0.45\textwidth]{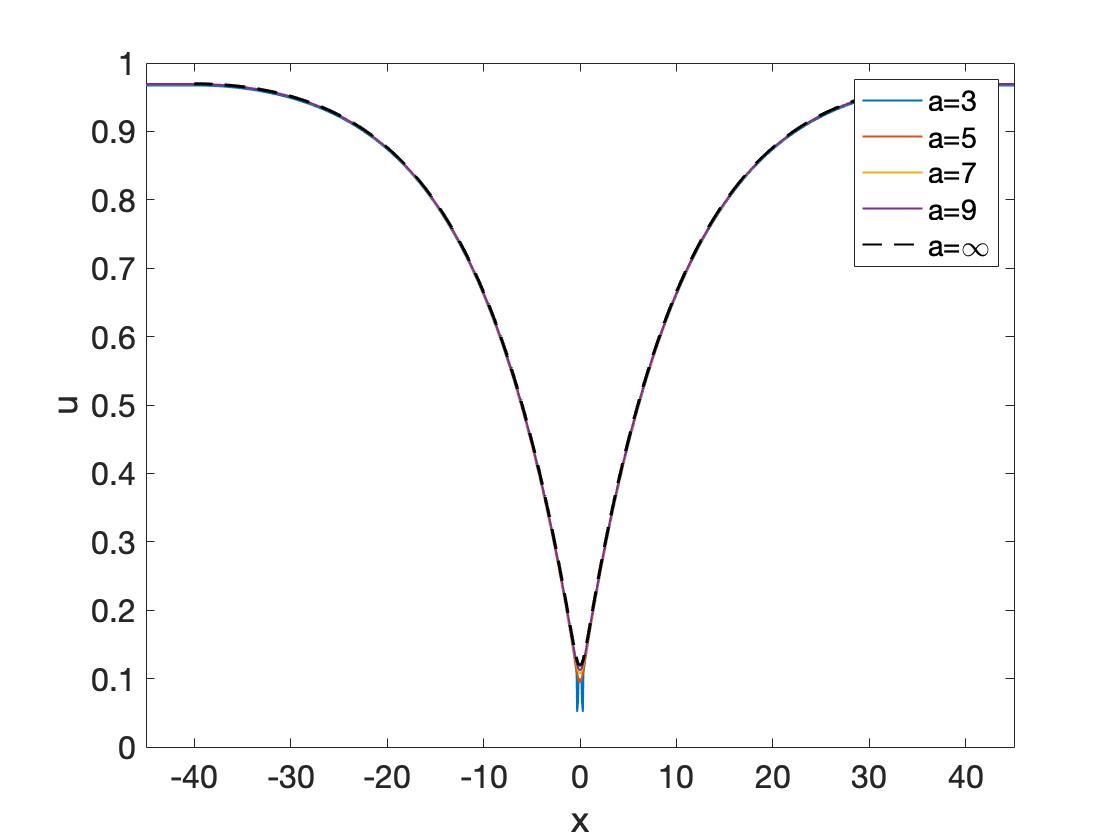} }
   \subfigure[]{\includegraphics[width=0.45\textwidth]{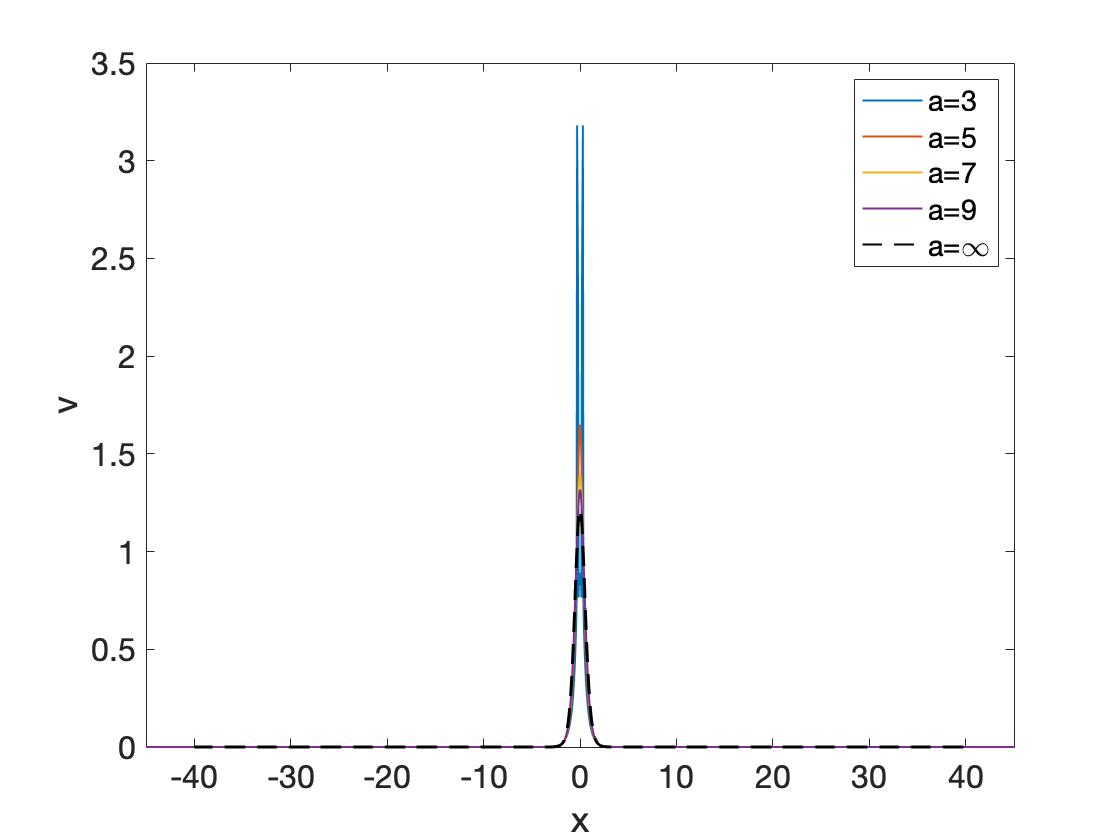} }
  \subfigure[]{ \includegraphics[width=0.45\textwidth]{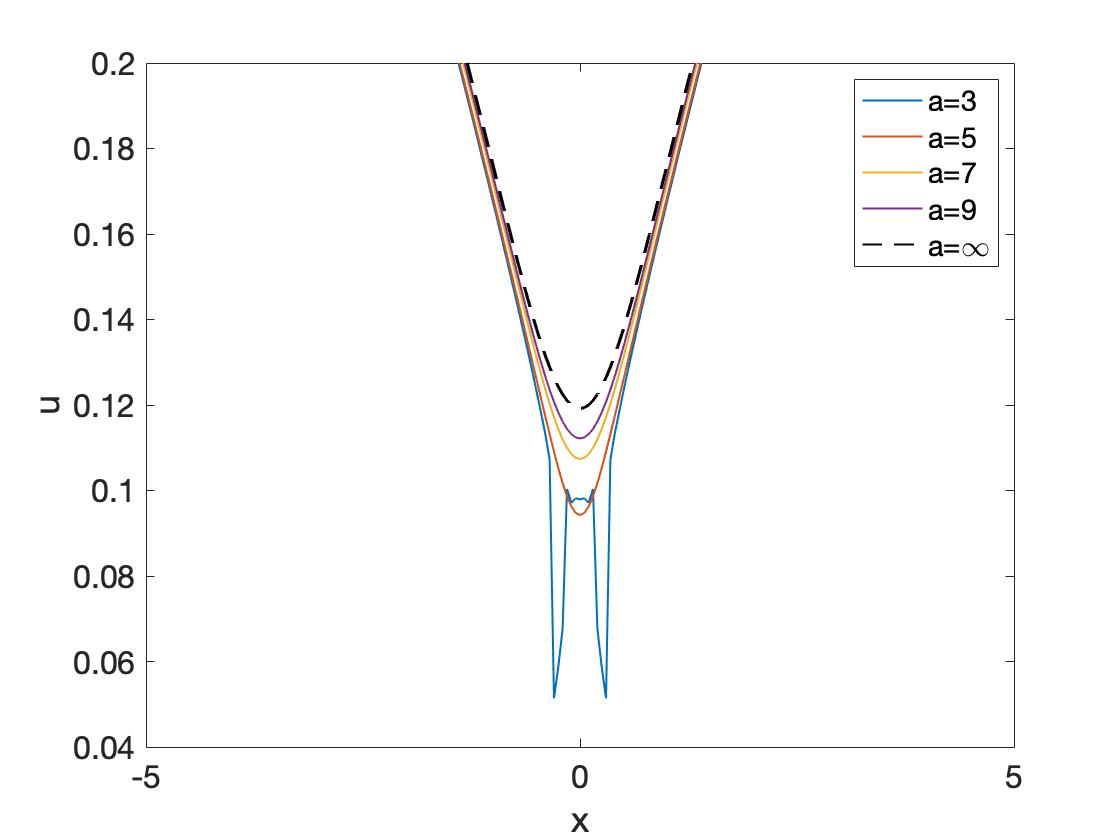} }
   \subfigure[]{\includegraphics[width=0.45\textwidth]{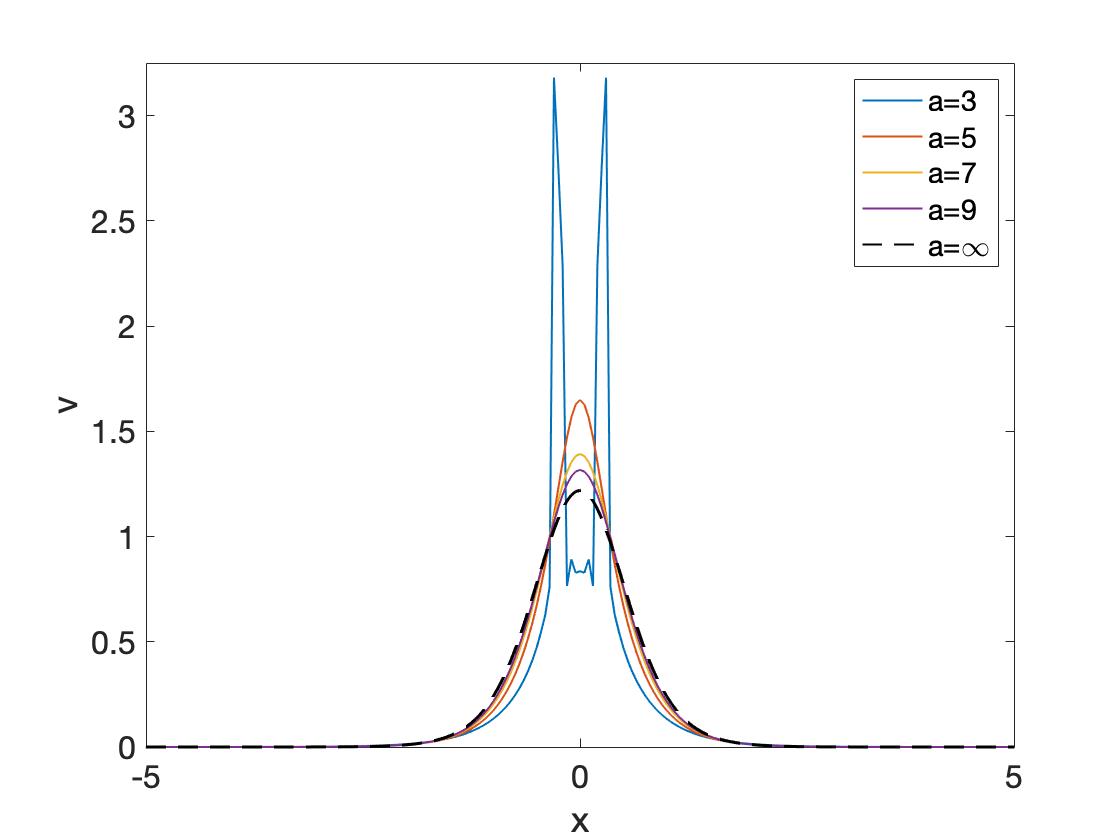} }
   \caption{Pulse solutions for the nonlocal Gray-Scott model for different values of the dispersive range 
   $a=\{ 3,5,7,9\}$ (solid lines) and for the local Gray-Scott model $a= \infty$ (dashed line). Figures a) and b) depict the $u$ and $v$ profiles for the solution, respectively. Figures c) and d) zoom in into a neighborhood of the origin. Curves that are closer to dashed line correspond to larger values of $a$. Other parameters used are specified in the text.}
   \label{fig:pulse_neumann}
\end{figure}

Thus, our results show that there is a critical value of the parameter $a$ that signals a transition from pulse solutions
to a different type of pattern. While our numerical results suggest that these new patterns exhibit a profile with two peaks, it is not clear if this steady-state is the result of the boundary conditions used, or if one could find these type of solutions when considering the problem on the whole real line.
Indeed, this double-peaked pattern could be the result of the single pulse solutions beginning to split into
two, but not having enough room to separate. 
Preliminary results show that when the computational domain is extended  to $\Omega = [-50,50]$, the solution
still converges to the same profile near the center of the domain, see Figure \ref{fig:aa_3}. 
\blue{Similarly, when refining the mesh from $h =0.05$ to $h= 0.025$ and then to $h=0.0125$, we observed that the solution converged to the same profile as the  one shown in the Figure \ref{fig:aa_3}.}
This suggest that the `batman'-like pattern is indeed a steady state solution. 

\begin{figure}[ht] 
   \centering
  \subfigure[]{ \includegraphics[width=0.45\textwidth]{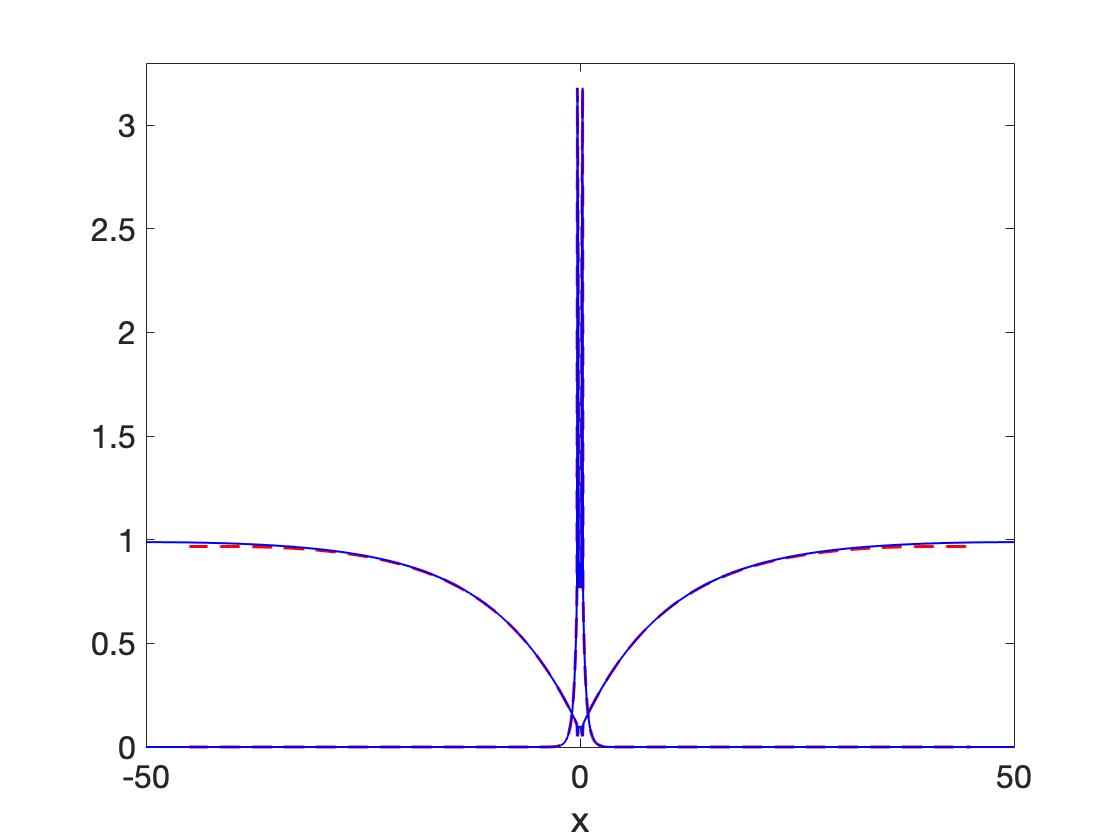} }
   \subfigure[]{\includegraphics[width=0.45\textwidth]{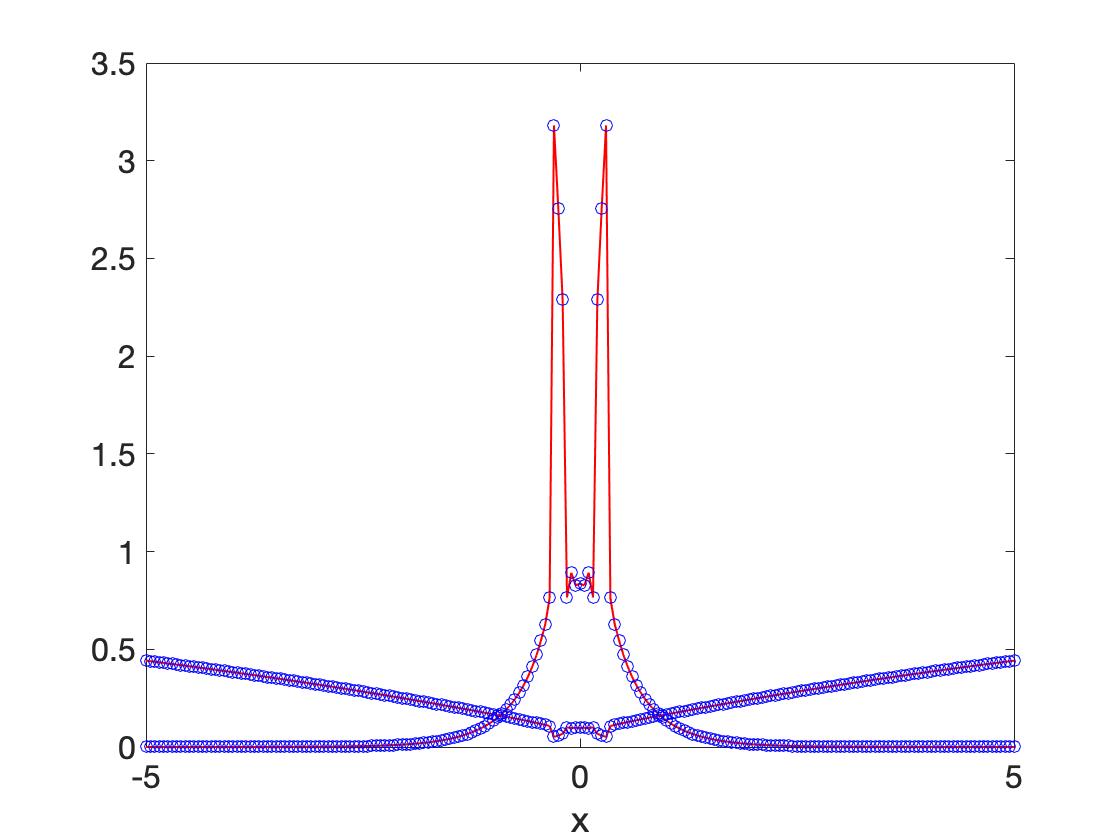} }
  \caption{Pulse solutions for the nonlocal Gray-Scott model with $ a= 3$.
   Figure a) depicts the $u$ and $v$ profiles of the solution when $\Omega = [-40,40]$ (dashed-red line) and $\Omega = [-50,50]$ (solid-blue line). Figures b) zooms in into a neighborhood of the origin, where the solid-red line corresponds to $\Omega = [-40,40]$    and the blue circles to $\Omega= [-50,50]$. }
   \label{fig:aa_3}
\end{figure}

\section{Discussion}
\label{s:discussion}

In this paper we studied a nonlocal Gray-Scott model 
where diffusion is replaced by long-range dispersal.
To describe  these nonlocal interactions we considered 
\blue{a nonlocal diffusion operator $K$}
that uses  an $L^1$ symmetric, positive kernel.
We looked at the equations posed on a bounded domain with
nonlocal homogeneous Dirichlet and Neumann boundary constraints.
In the Dirichlet case we used a spatially extended kernel,
while in the Neumann case we assumed the kernel has compact support.
We proved the existence of weak solutions and constructed
 numerical schemes for both types of boundary conditions.
Finally, we tested our algorithm by investigating the role that nonlocal diffusion has
in shaping pulse solutions in the nonlocal Gray-Scott model.

To our knowledge, our results represent the first proof for the existence of small-time
weak solutions for a system of integro-differential equations with cubic nonlinearities and nonlocal diffusion.
Because the bilinear form generated by the nonlocal operator has as its domain a
Sobolev space with little regularity, it is not possible to immediately apply  the 'standard' Galerkin approach
that is generally used
to prove existence of weak solutions of nonlinear parabolic equations.
To overcome this difficulty, the method presented here relied on extending the system of equations 
to include expressions for two additional unknowns. We showed that at the level of the Galerkin approximation
these  unknowns are the same as the derivatives of the original variables, but only on the domain of interest, $\Omega$, and not necessarily on the extended domain $\tilde{\Omega}$.
This allowed us to obtain further regularity for the solutions \blue{on $\Omega$}, which we could then use
in combination with Aubin's compactness theorem to prove weak convergence of the nonlinear terms.

To construct our algorithms we used finite elements for the space discretization and an implicit Euler scheme for the time stepping.
The proposed numerical method follows the approach taken by Du et al \cite{du2012, du2013},
where a notion of flux that is compatible with the nonlocal operator is defined. This
in turn allows one to properly define nonlocal Dirichlet and Neumann boundary constraints.
More precisely, to implement homogeneous nonlocal Dirichlet constraints we set the value of the unknowns to zero 
outside the computational domain, while for homogeneous Neumann constraints we required
that the unknowns, $u,v$, satisfy
as $\blue{K u}=0$ or $\blue{K v}=0$ on the outer domain $\Omega_0$.
To test the numerical schemes we ran non-trivial numerical experiments and showed that
both algorithms have an order of convergence equal to one.

Our motivation for considering the nonlocal Gray-Scott equations was two fold. 
First, the original Gray-Scott equations are known to give rise to interesting 
patterns that include periodic structures, pulse solutions, multi-pulse solutions, 
and other interesting spatio-temporal patterns. Our long term goal here is to understand
how nonlocal diffusion and the associated nonlocal boundary constraints affect the formation of these structures. 
In this paper we investigated only pulse solutions when the problem is assumed to have Neumann boundary constraints.
Our numerical results showed that the dispersive range, $a$, of the nonlocal operator affects
the shape of the pattern. As the value of $a$ decreases (corresponding to an 
increase in the dispersive range of the operator), solutions exhibit sharper, taller peaks.
Our simulations also suggest that there is a critical value of the parameter $a= a_c$ 
such that if $a$ falls below this value, the solution is no longer a single pulse, but 
exhibits instead a `batman' like profile. This pattern seems to be a new steady state solution,
not found in the local case. However, a detailed analysis of the nonlocal equations is necessary to prove 
the existence of these structures.

{ The above discussion suggests a number of future directions to consider.
Because the discretization of the nonlocal operator results in a dense matrix,
numerical simulations can take up significant time.
This can be partially remedied by
implementing adaptive mesh algorithms or by using nonuniform mesh. Because pulse solutions
are characterized by two different spatial scales, with wide regions where the pattern
changes slightly mixed with small  areas where the solution changes rather quickly, by picking a suitable adaptive mesh these features can be exploited to reduce the system size.

 In addition, because pulse solutions, as well as other
spatio-temporal patterns, constitute steady states of the equation, finding these solutions
by evolving the system of equations can take very long computational times.
 Therefore, the approach used here
 is not the most efficient  method for investigating
bifurcations as parameters vary. 
Future work should focus on developing continuation methods for nonlocal 
problems that account as well for nonlocal boundary constraints. 
These methods, which are based on the implicit function theorem,
are specifically designed to track steady-state solutions of ordinary differential equations
as parameters are varied. Examples of software packages designed for this purpose include AUTO-07P \cite{doedel2007}, and MatCon \cite{matcont}.
In the past decades this approach has also been adapted for finding steady-states of hyperbolic partial differential equations \cite{pde2path}, and for integro-differential equations with local or periodic boundary conditions \cite{avitabile2020, helmut2020, rankin2014, sherratt2019}.
However, to our knowledge there are no continuation methods for nonlocal operators with 
the types of nonlocal boundary constraints considered here.
 We intend to fill this void in future work.

Our second motivation for considering the Gray-Scott equations
comes from their connection to
the generalized Klausmeier model \cite{doelman2013}, a system of reaction diffusion equations
that is used to describe vegetation patterns.
In particular, it is known that by a suitable change of coordinates, the spatially homogeneous steady-states of the Gray-Scott 
equations can be mapped to those of the Klausmeier model. Both systems also give rise to similar patterns. Indeed the extended generalized Klausmeier model
is able to support periodic patterns \cite{doelman2013} as well as pulse and multi-pulse solutions \cite{sewalt2017}.
However, as mentioned in the introduction, a disadvantage of these equations is their use 
of the Laplace operator to model seed dispersal.
While this modeling  choice  still provides valuable insight into the formation of vegetation patterns, 
  it is now widely accepted that nonlocal convolution operators provide a better description
  of  dispersal effects in biological applications
  \cite{bullock2017}. Indeed,  a number of vegetation models 
  use this assumption to describe seed dispersal when studying
  pattern formation, \cite{pueyo2008, pueyo2010, baudena2013}, or when looking at the spread, or invasion, of plants
  \cite{allen1996, powell2004}. 
  It would be interesting to  study how nonlocal boundary constraints affect pattern formation in these
  systems. We plan to address this challenge in future work.
  }


\section{Appendix}
A proof of the following result can be found in \cite{aubin1963}
\begin{Theorem}[Aubin's Compactness Theorem]\label{t:Aubin}
Let $X_0,X, X_1$ be three Banach spaces such that $$X_0 \subset X \subset X_1,$$ where the injections are continuous,
$X_0$ and $X_1$ are both reflexive, and the injection $X_0 \longrightarrow X$ is compact.
Let $T>0$ be a fixed finite number, let $1< p,q< \infty$, and define
\[ \mathcal{Y} :=\{ v \in L^p(0,T;X_0): v' =\frac{dv}{dt} \in L^q(0,T;X_1)\}.\]
Then, $\mathcal{Y}$ is a Banach space when it is equipped with the norm
\[ \|v\|_\mathcal{Y} = \| v\|_{L^p(0,T; X_0)} + \| v'\|_{L^q(0,T; X_1)}.\]
Furthermore, the injection $\mathcal{Y}$ into $L^p(0,T;X)$ is compact.
\end{Theorem}

\bibliographystyle{plain}
\bibliography{nonlocal_bib_all}

\end{document}